\RequirePackage{fix-cm}
\documentclass[smallextended]{svjour3}       
\smartqed  
\usepackage[]{graphicx}
\usepackage{subfig}
\usepackage{graphicx}
\usepackage{amsmath,amssymb,amsfonts}%

\usepackage{array,multirow,graphicx}
\newtheorem{alg}{Algorithm}

\newtheorem{assumption}{Assumption}

\usepackage{caption}
\usepackage{subcaption}

\journalname{JOTA}
\begin{document}

\title{Adaptive 
	generalized conditional gradient method for multiobjective optimization 
}


\author{Anteneh Getachew  Gebrie      \and
  Ellen Hidemi Fukuda
}


\institute{A.G. Gebrie (corresponding author) \at
 Perimeter Institute for Theoretical Physics, Waterloo, ON, Canada\\
 Graduate School of Informatics, Kyoto University, 606–8501, Kyoto,
Japan \\
        \email{antgetm@gmail.com}         \and
         E.H. Fukuda  \at
   Graduate School of Informatics, Kyoto University, 606–8501, Kyoto,
Japan \\
\email{ellen@i.kyoto-u.ac.jp}\vspace{2mm}\\
This work is supported by Matsumae International Foundation of Japan and Grant-in-Aid for Scientific Research (C) (19K11840) from Japan Society for the Promotion of Science.
}

\date{Received: date / Accepted: date}

\maketitle

\begin{abstract}
In this paper, we propose a generalized conditional gradient method for multiobjective optimization, where the objective function is the sum of a smooth function and a possibly nonsmooth function. The proposed method is an improved extension of the classical Frank-Wolfe method of single-objective optimization to the multiobjective optimization problem. The method combines the so-called normalized descent direction as an adaptive procedure and the line search technique. We prove the convergence of the algorithm with respect to Pareto optimality under mild assumptions. The iteration complexity for obtaining an approximate Pareto critical point and the convergence rate in terms of a merit function is also analyzed.   Finally, we report some numerical results, which demonstrate the feasibility and competitiveness of the proposed method.
\end{abstract}

\section{Introduction}\label{sec1}

Several problems that appear in engineering, economics, management science, and medicine are modeled as multiobjective optimization problems~\cite{AAAA1,AAAA2}. Multiobjective optimization is the problem of finding a vector of decision variables that satisfies constraints (if they exist) while optimizing a vector-valued function in terms of (weakly) Pareto. The two commonly used types of methods for solving such problems are the {\it scalarization} techniques~\cite{AA2,AA1} and the {\it metaheuristics}~\cite{AA3}. 

Multiobjective optimization has received much attention over the past few decades due to its wide range of practical applications. Recently, proximal methods, gradient-type methods, and subgradient methods have been developed for multiobjective optimization problems by extending their single-objective optimization counterparts~\cite{2,3a,4,1,3}. They are often called descent methods because they decrease one or more objective functions within some iterations. Among these approaches, first-order, second-order, and derivative free methods have been proposed; see~\cite{4EEE,4DDD,4CCC} and references therein. For instance, the steepest descent method and the proximal point method for unconstrained multiobjective problems were first proposed in~\cite{4ab} and~\cite{4a}, respectively. The multiobjective proximal gradient method was also introduced in~\cite{4aa} for problems with composite objectives, and studied in~\cite{22eliant,20pomyant,21rakant,13,12,4aa}.

Moreover, the {\it Frank-Wolfe}~\cite{5} ({\it conditional gradient}) method proposed for convex continuously differentiable optimization problems, has been modified and extended to single-objective~\cite{8,7,6} and multiobjective optimization~\cite{2,9,1} in several ways. The conditional gradient method has attracted the attention of many researchers and has grown in popularity because it is often less computationally expensive than proximal-type methods and can solve large-scale problems.
Recently, in~\cite{8}, the conditional gradient method was presented by defining the ($\varepsilon$-normalized) descent direction method for single-objective optimization problems with convex and continuously differentiable objective functions. This method combines the ideas of the gradient method and the descent direction method for smooth single-objective optimization. To the best of our knowledge, there is no work related to this normalized descent direction approach~\cite{8} in the multiobjective settings or even in the single-objective case with more general objective functions.

Here, we show that~\cite{8} can be interpreted as a generalized conditional gradient method that allows the exploration of Pareto optimal points for multiobjective optimization problems. By adopting some ideas of~\cite{8}, we introduce a new descent search direction, using an adaptive method with line search, which provides a new generalized conditional gradient method for convex multiobjective optimization,  where each objective function is the sum of a smooth and a possibly nonsmooth vector-valued functions. Unlike the methods in \cite{2,1}, the adaptive step size selection strategy needs only estimates rather than the exact values of the Lipschitz constants, which further improves the implementation process of the proposed method. 
 
The major contribution of our work is presenting a novel adaptive generalized conditional gradient method with Armijo line search for composite multiobjective optimization. The implementation of the algorithm does not necessarily require the exact values of the Lipschitz constants; their estimate values are enough.  The procedures in the algorithm ensure that every accumulation point of the sequence generated by the method is a Pareto optimal point. Furthermore, we prove the iteration complexity and the convergence rate of our method under reasonable assumptions and by using a merit function~\cite{10aabbcc}. We show that the algorithm finds the first $\mu$-approximate Pareto critical point within ${\bf\it {O}}(1/\mu^{2})$ number of iterations, and the convergence rate is ${ {\bf\it O}}(1/k)$ where $k$ is the number of iterations.

The paper is organized as follows. Section~\ref{Sec2} presents the definition and basic properties of some important concepts. The normalized descent direction of the vector-valued objective function is introduced and characterized in Section~\ref{Sec22}. In Section~\ref{Sec3}, we present the proposed algorithm, and Section~\ref{Sec4} contains its convergence analysis.  Finally, in Section~\ref{Sec5}, some numerical experiments are presented to demonstrate the efficiency of our method.

\section{Preliminaries}\label{Sec2}
Let $\mathbb{R}^{n}$ be the $n$-dimensional Euclidean space. For $x,y\in\mathbb{R}^{n}$, we say that $x\preceq y$ ($y\succeq x$) if $y-x\in\mathbb{R}^{n}_{+}=\{z\in\mathbb{R}^{n}:z_{i}\geq 0\hspace{2mm}  \forall i=1,\ldots,n\}$, and $x\prec y$ ($y \succ x$) if $y-x\in\mathbb{R}^{n}_{++}=\{z\in\mathbb{R}^{n}:z_{i}> 0\hspace{2mm}  \forall i=1,\ldots,n\}$. We use $\top$ to denote the transpose of a matrix, and $ e_{n}=(1,\dots,1)^{\top}\in\mathbb{R}^{n}$.

Let us now review some definitions, important concepts, and properties of convex extended real-valued and vector-valued functions. For a complete review, see the references by Bauschke et al.~\cite{14} and Miettinen~\cite{15}, for extended real-valued and vector-valued functions, respectively.
Take a convex function  $h \colon \mathbb{R}^{n}\rightarrow\mathbb{R}\cup\{+\infty\}$. Then, the domain of $h$ is denoted by dom$(h)=\{x\in\mathbb{R}^{n} \colon h(x)<+\infty\}$. For $x\in\mbox{dom}(h)$ and $d\in\mathbb{R}^{n}$, the directional derivative of $h$ at a point $x$ along the direction $d$ is given by
$$h'(x;d):=\lim_{t\downarrow 0}\frac{h(x+td)-h(x)}{t}.$$	
It is easy to
see that $h'(x;d)=\langle\nabla h(x),d\rangle$ when $h$ is differentiable at $x\in \mbox{int(dom}(h))$, where $\nabla h(x)$ denotes the
gradient of $h$ at $x$. For $x\in \mbox{dom}(h)$ and $d\in\mathbb{R}^{n}$, define $\bar{h}\colon(0,+\infty)\rightarrow\mathbb{R}\cup\{+\infty\}$ by
$\bar{h}(t)=(h(x+td)-h(x))/t.$	
The function $\bar{h}$	is non-decreasing, and hence,
$h(x+d)-h(x)\geq(h(x+td)-h(x))/t$
for all $t\in(0,1)$. Moreover, this gives 	$h(x+d)-h(x)\geq h'(x;d)$. If $h=f+g$, where
$g:\mathbb{R}^{n}\rightarrow\mathbb{R}\cup\{+\infty\}$  is convex proper and lower semicontinuous function and  $f:\mathbb{R}^{n}\rightarrow\mathbb{R}$ is convex and differentiable function, then,
$h'(x;d)=\langle\nabla f_{i}(x),d\rangle+g'_{i}(x,d)$ for $x\in \mbox{dom}(f)$ and $d\in\mathbb{R}^{n}$. For $\mathcal{F}=(f_{1},\ldots,f_{m})$ the Jacobian matrix of $\mathcal{F}$ is the
$m\times n$ matrix defined by
$$J\mathcal{F}(x):=[\nabla f_{1}(x),\ldots,\nabla f_{m}(x)]^{\top}.$$	

Now, consider the constrained multiobjective optimization problem: 
\begin{eqnarray}\label{prob}
&&\min \: \mathcal{V}(x):=\mathcal{F}(x)+\mathcal{G}(x)\\&&\hspace{1mm}\mbox{s.t. } \: x\in \Omega,\nonumber
\end{eqnarray}
where $\Omega$ is non-empty compact convex subset of $\mathbb{R}^{n}$, $\mathcal{V}\colon\mathbb{R}^{n}\rightarrow(\mathbb{R}\cup\{+\infty\})^{m}$, $\mathcal{F}:=(f_{1},\ldots,f_{m})$, and $\mathcal{G}:=(g_{1},\ldots,g_{m})$ with $f_{i}\colon\mathbb{R}^{n}\rightarrow\mathbb{R}$ and $g_{i}\colon\mathbb{R}^{n}\rightarrow\mathbb{R}\cup\{+\infty\}$. 
Let  $$\mathcal{V}_{i}:=f_{i} +g_{i},$$ where $f_{i}$ and $g_{i}$ ($i\in\{1,\ldots,m\}$) are functions given under (\ref{prob}), and therefore, $\mathcal{V}$ in (\ref{prob}) is given by $\mathcal{V}=
(\mathcal{V}_{1},\ldots,\mathcal{V}_{m})$. Throughout this paper, we make
the following assumptions.

\begin{assumption} \label{AA-111}
 \begin{description}
	\item[\textbf{(A1)}] $f_{i}$ is convex, continuously differentiable, and $\nabla f_{i}\colon\mathbb{R}^{n}\rightarrow \mathbb{R}^{n}$ is $L_{i}$-Lipschitz continuous for
	$i\in \{1,\ldots,m\}$.
	\item[\textbf{(A2)}]  $g_{i}$  is convex, proper, lower semicontinuous, but not necessarily differentiable, and $g_{i}$ is $L_{g_{i}}$-Lipschitz continuous in its domain for
	$i\in\{1,\ldots,m\}$.	
\end{description}
\end{assumption}  
A point $\bar{x}\in\Omega$ is called Pareto optimal of~\eqref{prob} if there is no other $x\in\Omega$ with $\mathcal{V}(x)\preceq \mathcal{V}(\bar{x})$ and $\mathcal{V}(x)\neq \mathcal{V}(\bar{x})$. Moreover, a point $\bar{x}\in\Omega$ is weakly Pareto optimal point of~\eqref{prob} if there is no other $x\in\Omega$ with $\mathcal{V}(x)\prec \mathcal{V}(\bar{x})$. It is clear that every Pareto optimal point is a weakly Pareto optimal point, but the converse is not true. 

Let $\bar{x}$ be
a weakly Pareto optimal point of (\ref{prob}). Then $\bar{x}$ must satisfy
$$(\mathcal{V}'_{1}(\bar{x};d),\ldots,\mathcal{V}'_{m}(\bar{x};d))\notin -\mathbb{R}_{++}^{n} \mbox{ for all } d\in\mathbb{R}^{n}.$$ This shows that
\begin{eqnarray}\label{prob33}\max_{i=1,\ldots,m}\mathcal{V}'_{i}(\bar{x};d)\geq 0 \quad \mbox{ for all } d\in\mathbb{R}^{n}.
\end{eqnarray}
The inequality in (\ref{prob33}) is sometimes referred to in the literature as the criticality condition
or the first-order necessary condition for weak efficiency of~\eqref{prob} and $\bar{x}\in \mbox{dom}(\mathcal{G})=\cap_{i=1}^{m}\mbox{dom}(g_{i})$ satisfying~\eqref{prob33}
is often called a Pareto stationary (critical) point of~\eqref{prob}. Furthermore, convexity of each $\mathcal{V}_{i}$, $i=1,\ldots,m$
ensures that every Pareto critical point of~\eqref{prob} is a weakly Pareto optimal point of~\eqref{prob}. Note that if each $\mathcal{V}_{i}$, $i=1,\ldots,m$ is a strictly or strongly convex function, then Pareto critical points of~\eqref{prob} are Pareto optimal. Gradient-based methods have been used to solve several multiobjective optimization problem. Gradient-based approaches, however, are not entirely capable of solving problems such as portfolio optimization \cite{19qq}, and new and generalized problems of the form~\eqref{prob} and their solution approaches are constantly needed for a better theoretical insights and handle variety of optimization problems that appear in applications, for example, in Machine Learning (deep neural network) \cite{16aaaaa} and portfolio optimization \cite{19qq,6aaaa}.

\section{Normalized Descent Search Directions}\label{Sec22}

For the multiobjective optimization (\ref{prob}), let  
\begin{equation}
  \label{eq:Lipschitz}
  \mathcal{D}:=\mbox{dom}(\mathcal{G})=\bigcap_{i=1}^{m}\mbox{dom}(g_{i}),\hspace{1mm}\hspace{0.5mm} L_{G}:=\max\limits_{i=1,\ldots,m}L_{g_{i}}, \hspace{0.5mm}\mbox{ and }\hspace{0.5mm} L:=\max\limits_{i=1,\ldots,m}L_{i}.
\end{equation}
We consider (\ref{prob}) under $\Omega\subset\mathcal{D}$.  Note that from the so-called descent lemma \cite{10}, we have
	\begin{eqnarray}	\label{descthey}
	\mathcal{F}(y)-\mathcal{F}(x)\preceq  J\mathcal{F}(x)(y-x)+\frac{L}{2}\|x-y\|^{2}e_{m} \mbox{ for all } (x,y)\in\mathcal{D}\times\mathcal{D}.	\end{eqnarray}

We now introduce our $\varepsilon$-normalized descent direction of the
objective function~$\mathcal{V}$, which can also be thought of as an extended and modified version of \cite[Definition 2.3]{8}. This definition and its related results will be used to establish our algorithm.
\begin{definition} ($\varepsilon$-normalized descent direction) Let $\varepsilon>0$. A non-zero vector $d\in\mathbb{R}^{n}$ is called an $\varepsilon$-normalized descent direction of $\mathcal{V}$ of the
	multiobjective optimization (\ref{prob}) at the point $x\in \Omega$ if the following holds:
	$$\max\limits_{i=1,\ldots,m}\big\{\langle \nabla f_{i}(x),d \rangle+g_{i}(x+d)-g_{i}(x)\big\}+\varepsilon\|d\|^{2}\leq 0.$$	
\end{definition}

The next two results show how to construct an $\varepsilon$-normalized descent direction given a point that satisfies some condition.

\begin{lemma}\label{lemS2-02aa} ($\varepsilon$-normalization) Let $u\in\mathbb{R}^{n}\setminus\{0\}$, $x\in \Omega$, and $\varepsilon>0$. If there is $t$ with $$\max\limits_{i=1,\ldots,m}\big\{\langle \nabla f_{i}(x),u \rangle+g_{i}(x+u)-g_{i}(x)\big\}+t\leq 0$$ and $0<t\leq\varepsilon\|u\|^{2}$,  
	then the vector $d:=\frac{t u}{\varepsilon\|u\|^{2}}$ is an $\varepsilon$-normalized descent direction of~$\mathcal{V}$ at~$x$.
\end{lemma}
\begin{proof} Since $0<\frac{t} {\varepsilon\|u\|^{2}}\leq1$ and $g_{i}$ is convex, we have	
	\begin{eqnarray}
	g_{i}\Big(x+\frac{t u}{\varepsilon\|u\|^{2}}\Big)\nonumber&=&g_{i}\Big(\frac{t }{\varepsilon\|u\|^{2}}(x+u)+(1-\frac{t }{\varepsilon\|u\|^{2}})x\Big)	\nonumber\\&\leq&\frac{t }{\varepsilon\|u\|^{2}}g_{i}(x+u)+\Big(1-\frac{t }{\varepsilon\|u\|^{2}}\Big)g_{i}(x).\nonumber
	\end{eqnarray}	
	Thus, using the above inequality we have
	\begin{eqnarray}
	\nonumber&&\max\limits_{i=1,\ldots,m}\big\{\big\langle \nabla f_{i}(x),d \big\rangle+g_{i}(x+d)-g_{i}(x)\big\}+\varepsilon\|d\|^{2}	\nonumber\\&&=\max\limits_{i=1,\ldots,m}\Big\{\Big\langle \nabla f_{i}(x),\frac{t u}{\varepsilon\|u\|^{2}} \Big\rangle+g_{i}\Big(x+\frac{t u}{\varepsilon\|u\|^{2}}\Big)-g_{i}(x)\Big\}+\varepsilon\Big\|\frac{t u}{\varepsilon\|u\|^{2}}\Big\|^{2}\nonumber\\&&\leq\max\limits_{i=1,\ldots,m}\Big\{\frac{t}{\varepsilon\|u\|^{2}}\Big(\langle \nabla f_{i}(x),u \rangle+g_{i}(x+u)-g_{i}(x)\Big)\Big\}+\frac{t^{2} }{\varepsilon\|u\|^{2}}\nonumber\\&&=\frac{t}{\varepsilon\|u\|^{2}}\Big(\max\limits_{i=1,\ldots,m}\Big\{\langle \nabla f_{i}(x),u \rangle+g_{i}(x+u)-g_{i}(x)\Big\}+t\Big)\leq 0,\nonumber
	\end{eqnarray}
    where the last inequality holds from the assumption of the lemma. 
	Therefore, $d$ is an $\varepsilon$-normalized descent direction of  $\mathcal{V}$ at~$x$.	\qed
\end{proof}

\begin{lemma}\label{lemS2-02}
	Let $u\in\mathbb{R}^{n}\setminus\{0\}$, $x\in \Omega$, and $\varepsilon>0$, where $u$ is not an $\varepsilon$-normalized descent direction of $\mathcal{V}$ at $x$. If $$\max\limits_{i=1,\ldots,m}\big\{\langle \nabla f_{i}(x),u \rangle+g_{i}(x+u)-g_{i}(x)\}<0,$$ then $d:=\frac{t u}{\varepsilon\|u\|^{2}}$ is an $\varepsilon$-normalized descent direction of $\mathcal{V}$ at $x$ for any real number $t$ such that $$0 <t\leq\Big|\max\limits_{i=1,\ldots,m}\big\{\langle \nabla f_{i}(x),u \rangle+g_{i}(x+u)-g_{i}(x)\}\Big|.$$\end{lemma}
\begin{proof}
	Here, $d=\frac{t u}{\varepsilon\|u\|^{2}}\neq 0$ and $t$ satisfies the  conditions $0<t<\varepsilon\|u\|^{2}$ and $\max\limits_{i=1,\ldots,m}\big\{\langle \nabla f_{i}(x),u \rangle+g_{i}(x+u)-g_{i}(x)\big\}+t\leq 0$. Hence, by Lemma \ref{lemS2-02aa}, we have $d=\frac{t u}{\varepsilon\|u\|^{2}}$ is an $\varepsilon$-normalized descent direction of $\mathcal{V}$ at~$x$.	\qed
\end{proof}

In the subsequent results, we use the following notation. Letting $\varepsilon>0$, we define
$$\zeta_{\varepsilon}:=
\begin{cases}
\frac{2\varepsilon}{L}, \hspace{4mm}\mbox{if}\hspace{2mm} \varepsilon<\frac{L}{2}, \\ 1, \hspace{6mm}\mbox{otherwise}.
\end{cases}
$$
From this definition, note that $\zeta_{\varepsilon} \in (0,1]$.

\begin{lemma}\label{lemS2-03}  Let $\varepsilon>0$, $\beta\in (0,\frac{1}{2})$, $d\in\mathbb{R}^{n}\setminus\{0\}$, and $x\in \Omega$. If $d$ is an $\varepsilon$-normalized descent direction of $\mathcal{V}$ at $x$, then
	\begin{eqnarray}\label{S2-1}\mathcal{V}(x+\eta d)-\mathcal{V}(x)\preceq \eta\beta\big(J\mathcal{F}(x)d+\mathcal{G}(x+d)-\mathcal{G}(x)-\varepsilon\|d\|^{2}e_{m}\big)
	\end{eqnarray}	
	for all $\eta\in(0,\bar{\eta}]$
	where $\bar{\eta}=(1-2\beta)\zeta_{\varepsilon}$.
\end{lemma}
\begin{proof} Let us denote $$\omega:=-\big(J\mathcal{F}(x)d+\mathcal{G}(x+d)-\mathcal{G}(x)\big)+\varepsilon\|d\|^{2}e_{m}$$  and $$\Lambda(\omega):=J\mathcal{F}(x)d+\frac{L\eta}{2}\|d\|^{2}e_{m}+\mathcal{G}(x+d)-\mathcal{G}(x)+\beta \omega.$$ Now, since $1-2\beta \in (0,1)$ and $\zeta_{\varepsilon} \in (0,1]$, we have $\eta\in (0,1)$. Thus, using the convexity of the function $\mathcal{G}$ and (\ref{descthey}), we obtain
	\begin{eqnarray}
	\mathcal{V}(x+\eta d)-\mathcal{V}(x)\nonumber&&=(\mathcal{F}(x+\eta d)-\mathcal{F}(x))+(\mathcal{G}(x+\eta d)-\mathcal{G}(x))	\nonumber\\&&=(\mathcal{F}(x+\eta d)-\mathcal{F}(x))+(\mathcal{G}((1-\eta)x+\eta (x+d))-\mathcal{G}(x))\nonumber\\&&\preceq \Big(\eta J\mathcal{F}(x)d+\frac{L\eta^{2}}{2}\|d\|^{2}e_{m}\Big)+(\eta\mathcal{G}(x+d)-\eta\mathcal{G}(x))\nonumber\\&&= \eta\Big(J\mathcal{F}(x)d+\frac{L\eta}{2}\|d\|^{2}e_{m}+\mathcal{G}(x+d)-\mathcal{G}(x)\Big)\nonumber\\&&= \eta\Lambda(\omega)-\eta\beta \omega.\nonumber
	\end{eqnarray}	
    Using the definition of $\omega$ and $\Lambda(\omega)$, we have
	\begin{eqnarray}
	\Lambda(\omega)\nonumber&&=(1-\beta)\big(J\mathcal{F}(x)d+\mathcal{G}(x+d)-\mathcal{G}(x)\big)+\beta\varepsilon\|d\|^{2}e_{m}+\frac{L\eta}{2}\|d\|^{2}e_{m}\\&&\nonumber=(1-\beta)\big(J\mathcal{F}(x)d+\mathcal{G}(x+d)-\mathcal{G}(x)+\varepsilon\|d\|^{2}e_{m}\big)\nonumber\\&&\hspace{38mm}-2(1-\beta)\varepsilon\|d\|^{2}e_{m}+\varepsilon\|d\|^{2}e_{m}+\frac{L\eta}{2}\|d\|^{2}e_{m}\nonumber\\&&\nonumber\preceq(1-\beta)\big(\max\limits_{i=1,\ldots,m}\big\{\langle \nabla f_{i}(x),d \rangle+g_{i}(x+d)-g_{i}(x)\}e_{m}+\varepsilon\|d\|^{2}e_{m}\big)\nonumber\\&&\hspace{38mm}+2(\beta-1)\varepsilon\|d\|^{2}e_{m}+\varepsilon\|d\|^{2}e_{m}+\frac{L\eta}{2}\|d\|^{2}e_{m}\nonumber\\&&\preceq2(\beta-1)\varepsilon\|d\|^{2}e_{m}+\varepsilon\|d\|^{2}e_{m}+\frac{L\eta}{2}\|d\|^{2}e_{m}\nonumber\\&&
    =\Big((2\beta-1)\varepsilon+\frac{L\eta}{2}\Big)\|d\|^{2}e_{m},\nonumber
	\end{eqnarray}
 	where the last inequality holds because $d$ is an $\varepsilon$-normalized descent direction  of~$\mathcal{V}$ at~$x$.
	Now assume that $(2\beta-1)\varepsilon+\frac{L\eta}{2}\leq 0$, or in other words, $0<\eta\leq \frac{2(1-2\beta)\varepsilon}{L}.$ Then, $\Lambda(\omega)\preceq 0$, and consequently,  $$\mathcal{V}(x+\eta d)-\mathcal{V}(x)\preceq -\eta\beta \omega=\eta\beta\big(J\mathcal{F}(x)d+\mathcal{G}(x+d)-\mathcal{G}(x)-\varepsilon\|d\|^{2}e_{m}\big).$$ Since
	$(1-2\beta)\zeta_{\varepsilon}\leq \frac{2(1-2\beta)\varepsilon}{L}$ from the definition of $\zeta_{\varepsilon}$, the result holds. \qed
\end{proof}	
\begin{lemma}\label{lemS2-03333} Let $\beta\in (0,\frac{1}{2})$, $d\in\mathbb{R}^{n}\setminus\{0\}$, and $x\in \Omega$. If $d$ is an $\varepsilon$-normalized descent direction of $\mathcal{V}$ at $x$, then
	$\mathcal{V}(x+\eta d)\prec \mathcal{V}(x)
	$ for  $0<\eta\leq(1-2\beta)\zeta_{\varepsilon}.$
\end{lemma}
\begin{proof} Since  $d$ is an $\varepsilon$-normalized descent direction of $\mathcal{V}$ at a point $x$, we have $\max\limits_{i=1,\ldots,m}\big\{\langle \nabla f_{i}(x),d \rangle+g_{i}(x+d)-g_{i}(x)\}+\varepsilon\|d\|^{2}\leq 0$. Therefore,  using Lemma \ref{lemS2-03} and noting that $d\neq0$ and $\beta,\varepsilon>0$, we get
	\begin{eqnarray}
	\mathcal{V}(x+\eta d)-\mathcal{V}(x)\nonumber&\preceq& \eta\beta\big(\max\limits_{i=1,\ldots,m}\big\{\langle \nabla f_{i}(x),d \rangle+g_{i}(x+d)-g_{i}(x)\}-\varepsilon\|d\|^{2}\big)e_{m}\nonumber\\&\preceq& -2\eta\beta\varepsilon\|d\|^{2}e_{m}\prec0\nonumber
	\end{eqnarray}	
	for all $\eta$ with $0<\eta\leq\bar{\eta}=(1-2\beta)\zeta_{\varepsilon}$. \qed
\end{proof}	
\begin{remark}\label{linesearch} Let  $d\in\mathbb{R}^{n}\setminus\{0\}$, $x\in \Omega$, and $d$ be an $\varepsilon$-normalized descent direction of $\mathcal{V}$ at $x$. Using Lemma~\ref{lemS2-03}, for each $\beta\in (0,\frac{1}{2})$, (\ref{S2-1}) holds for all $\eta$ where $0<\eta\leq (1-2\beta)\zeta_{\varepsilon}$. But, $(1-2\beta)\zeta_{\varepsilon}\leq(1-2\beta)<1$, and thus, there exists $j_{0}\in \mathbb{N}$ such that  $(1-2\beta)^{j}\leq (1-2\beta)\zeta_{\varepsilon}$ for all $j\in\{j_{0},j_{0}+1,j_{0}+2,j_{0}+3,\ldots\}$. 
	Hence, for each $\beta\in (0,\frac{1}{2})$, there exists $j_{0}\in \mathbb{N}$ such that (\ref{S2-1}) holds when $\eta=(1-2\beta)^{j}$ for all $j\in\{j_{0},j_{0}+1,j_{0}+2,j_{0}+3,\ldots\}$. Let $j^{*}$ be the smallest index $j_{0}\in \mathbb{N}$ such that
	\begin{eqnarray}\label{rakpmeli}
	\nonumber&&\mathcal{V}(x+(1-2\beta)^{j} d)-\mathcal{V}(x)\nonumber\\&&\hspace{10mm}\preceq(1-2\beta)^{j}\beta \big(J\mathcal{F}(x)d+\mathcal{G}(x+d)-\mathcal{G}(x)-\varepsilon\|d\|^{2}e_{m}\big)
	\end{eqnarray}
	for all $j\in\{j_{0},j_{0}+1,j_{0}+2,j_{0}+3,\ldots\}$. Note that, in view of Lemma \ref{lemS2-03},  if $2\varepsilon\geq L$, then $j^{*}=1$. The above inequality is an Armijo-type condition and will be used in our proposed algorithm. 
\end{remark}\vspace{2mm}	

\begin{lemma}\label{lemS2-04} Let $L>2\varepsilon>0$, $\beta\in(0,\frac{1}{2})$,  $d$ be an $\varepsilon$-normalized descent direction of $\mathcal{V}$ at $x$, and $j^{*}$ denotes the smallest positive integer such that (\ref{rakpmeli}) holds for $x$, $d$, $\varepsilon$, and~$\beta$.
	Then, we have $$0<\frac{2\varepsilon(1-2\beta)^{2}}{L}<(1-2\beta)^{j^{*}},$$
	that is, the estimate
	$2\varepsilon(1-2\beta)^{2-j^{*}}< L$ holds.
\end{lemma}
\begin{proof}If $j^{*}=1$, then it is straightforward. Let $j^{*}>1$. Now using the definition of $j^{*}$ it holds that 
	$$(1-2\beta)^{j^{*}-1}>(1-2\beta)\zeta_{\varepsilon}=\frac{2(1-2\beta)\varepsilon}{L}.$$ This implies that $(1-2\beta)^{j^{*}}>\frac{2(1-2\beta)^{2}\varepsilon}{L}.$ \qed
\end{proof}
\section{The Proposed Algorithm and Its Properties}\label{Sec3}
In this section, we will introduce an algorithm that combines the generalized conditional gradient, adaptive descent search direction, and Armijo line search to solve the multiobjective optimization problem~(\ref{prob}) satisfying Assumption~\ref{AA-111}.  We also present important properties of the proposed method which will be useful for establishing its convergence analysis.

Let us now introduce the function $\psi\colon\Omega\times\Omega\rightarrow\mathbb{R}\cup\{+\infty\}$  given by
\begin{eqnarray} \label{S3-001} \psi(x,y):=\max\limits_{i=1,\ldots,m}\{\langle\nabla f_{i}(x),y-x\rangle +g_{i}(y)-g_{i}(x)\},
\end{eqnarray} 
and $\theta\colon\Omega\rightarrow\mathbb{R}\cup\{-\infty\}$ defined by
\begin{eqnarray}\label{S3-002}
\theta(x):=\min_{y\in\Omega}\psi(x,y).
\end{eqnarray}
Also, the merit function $u_{0}\colon\Omega\rightarrow\mathbb{R}\cup\{+\infty\}$ defined in \cite{10aabbcc} is given by
\begin{eqnarray}\label{merit}
	u_{0}(x):=\sup_{y\in\Omega}\min_{i=1,\ldots,m}\big(\mathcal{V}_{i}(x)-\mathcal{V}_{i}(y)\big).\end{eqnarray}
\begin{lemma}\label{LmS3-02}
	Let $\psi$, $\theta$, and $u_{0}$ be defined by (\ref{S3-001}),  (\ref{S3-002}), and (\ref{merit}), respectively. Then the following statements hold.
	\begin{description}
	\item[(a).] $J\mathcal{F}(x)(y-x)+\mathcal{G}(y)-\mathcal{G}(x)\preceq\psi(x,y)e_{m}$ and $\theta(x)\leq\psi(x,y)$ for all $x,y\in\Omega$.
		\item[(b).]$\theta(x)\leq 0$ for all $x\in\Omega$. Moreover, $x\in\Omega$ is a Pareto critical point of (\ref{prob}) if and only if $\theta(x)=0$.
		
	\item[(c).]	$u_{0}(x)\geq 0$ for all $x\in\Omega$. Moreover, $x\in\Omega$ is weakly Pareto optimal
	of (\ref{prob}) if and only if $u_{0}(x)=0$.

 \item[(d).] For $\sigma\in [0,1]$ and $x,z\in \Omega$, $$\psi(x,x+\sigma (z-x))\leq\sigma \psi(x,z).$$
	\end{description}
\end{lemma}
\begin{proof}
	The proof of $(a)$ is straightforward. The result $(b)$  follows from \cite[Theorem 3.9]{10aabbcc} since its proof holds even if $\ell=0$ for the merit function $w_{\ell}(x)$ defined in there. Finally, $(c)$ is immediate from \cite[Theorem~3.1]{10aabbcc}.
 
From the definition of \eqref{S3-001} and $g_{i}(x+\sigma (z-x))\leq (1-\sigma)g_{i}(x)+\sigma g_{i}(z)$ (by convexity of $g_i$), we have 
 	\begin{eqnarray}
 	\psi(x,x+\sigma (z-x))\nonumber&\leq&\max\limits_{i=1,\ldots,m}\{\sigma\langle\nabla f_{i}(x), z-x\rangle +\sigma g_{i}(z)-\sigma g_{i}(x)\}\nonumber\\&=&\sigma \psi(x,z)\nonumber
 	\end{eqnarray}	
 	for all $x,z\in \Omega$ and $\sigma\in [0,1]$.\qed
 \end{proof}

We are now ready to state our line search-based adaptive generalized conditional gradient method ({\bf A-GCG}).

\par\noindent\rule{\textwidth}{1pt}
\begin{alg}\label{algorithm1}
{ \bf Adaptive  Generalized Conditional Gradient ({\bf A-GCG})} \par\noindent\rule{\textwidth}{1pt} Let  $\sigma_{\min}$, $\alpha_{\min}$, $\varepsilon_{0}$, and $\beta$ be real numbers with $0<\sigma_{\min}\leq 1$, $0<\alpha_{\min}$, $0<2\varepsilon_{0}<L$, and $0<\beta<\frac{1}{2}$. Choose $x_{0}\in\Omega$ and the sequences of real numbers $\{\sigma_{k}\}_{k=0}^{\infty}$ and $\{\alpha_{k}\}_{k=0}^{\infty}$ with $\sigma_{\min}\leq\sigma_{k}\leq 1$ and $\alpha_{\min}\leq\alpha_{k}$ for all $k$.  
	\begin{description}
		\item [\textbf{\underline {STEP 1.}}] Choose $y_{k}\in\Omega$ such that
		\begin{eqnarray}
		\psi(x_{k},y_{k})\leq\max\big\{\sigma_{k}\theta(x_{k}),-\alpha_{k}\big\}.\nonumber
		\end{eqnarray}
		
		If $\psi(x_{k},y_{k}):=0$, then stop. Otherwise, go to \textbf{STEP 2}.	
		\item [\textbf{\mbox{\underline {STEP 2.}} }] {\bf Adaptive  descent search direction technique:} Compute
		\[d_{k}:=
		\begin{cases}
		y_{k}-x_{k}, \hspace{12mm}\mbox{if}\hspace{1.5mm} \psi(x_{k},y_{k})\leq -\varepsilon_{k}\|y_{k}-x_{k}\|^{2}, \\ 
		\frac{t_{k}(y_{k}-x_{k})}{\varepsilon_{k}\|y_{k}-x_{k}\|^{2}}, \hspace{7mm}\mbox{otherwise},
		\end{cases}
		\]	
		where $t_{k}:=\big|\psi(x_{k},y_{k})\big|.$
		\item [\textbf{\underline {STEP 3.}}] {\bf Armijo-type line search technique:} Find $l_{k}$ where
		it is the smallest positive integer $j$ such that 
		\begin{eqnarray}\label{aaa22}
		\nonumber&&\mathcal{V}(x_{k}+(1-2\beta)^{j} d_{k})-\mathcal{V}(x_{k})\nonumber\\&&\hspace{3mm}\preceq (1-2\beta)^{j}\beta (J\mathcal{F}(x_{k})d_{k}+\mathcal{G}(x_{k}+d_{k})-\mathcal{G}(x_{k})-\varepsilon_{k}\|d_{k}\|^{2}e_{m}).
		\end{eqnarray}
		Set $\lambda_{k}:=(1-2\beta)^{l_{k}}$.
		
		\item [\textbf{\underline {STEP 4.}}] Compute $x_{k+1}:=x_{k}+\lambda_{k}d_{k}$ and update $\varepsilon_{k+1}:=\varepsilon_{k}(1-2\beta)^{1-l_{k}}$.\\ Set $k:=k+1$ and go to \textbf{STEP 1}.
	\end{description}	
\end{alg}
\par\noindent\rule{\textwidth}{1pt}
\vspace{0.3mm}

The difference between Algorithm \ref{algorithm1} ({\bf A-GCG}) and other multiobjective methods found in the literature (see for e.g., \cite{2,9,12,4aa} and references therein) is substantial because every iterative step in our algorithm employs both the adaptive step size and the Armijo line search technique, as well as relaxed parameters that can be adjusted and controlled. 
  \begin{remark} \label{remark2}
  \begin{description}
      \item[(a).] The computation to choose $y_{k}$ in {\bf STEP 1} of {\bf A-GCG} is technical and basically involves solving the subproblem (\ref{S3-002}) with $x=x_{k}$. Let $z_{k}\in\Omega$ with
$\psi(x_{k},z_{k})=\theta(x_{k})$. Now since $\theta(x_{k})\leq 0$ (by Lemma \ref{LmS3-02} $(b)$) and $0<\sigma_{\min}\leq\sigma_{k}\leq 1$, we have
$$\psi(x_{k},z_{k})=\theta(x_{k})\leq\sigma_{k}\theta(x_{k})\leq\max\big\{\sigma_{k}\theta(x_{k}),-\alpha_{k}\big\}.$$ 
      Moreover, 
for $\sigma>0$ with $\sigma_{\min}\leq\sigma_{k}\leq\sigma\leq 1$, the point $q_{k}=x_{k}+\sigma(z_{k}-x_{k})$  satisfies  
      $$
 \psi(x_{k},q_{k})\leq \sigma\psi(x_{k},z_{k})\leq \sigma_{k}\psi(x_{k},z_{k})=\sigma_{k}\theta(x_{k})\leq\max\big\{\sigma_{k}\theta(x_{k}),-\alpha_{k}\big\}.$$
 where the first inequality is obtained by Lemma \ref{LmS3-02} $(d)$.
 Hence, $y_k=z_{k}$ and $y_k=q_{k}$ satisfies {\bf STEP~1}, and thus
      $$\big\{z\in\Omega:\psi(x_{k},z)\leq\max\big\{\sigma_{k}\theta(x_{k}),-\alpha_{k}\big\}\big\}\neq\emptyset.$$ Note also that, if $\sigma_{k}\theta(x_{k})\leq -\alpha_{k}$, then any $y_k\in\Omega$ (if exists) with $\sigma_{k}\theta(x_{k})\leq\psi(x_k,y_k)\leq -\alpha_{k}$ also satisfies {\bf STEP~1}. Such $y_k$ possibly exists for $\sigma_{k}=1$ (or $\sigma_{k}$ very large close to 1) and for $\alpha_{k}=\alpha_{\min}$ (or $\alpha_{k}$ is very small close to $\alpha_{\min}$).
 Therefore, {\bf STEP~1} of {\bf A-GCG} is well-defined. 
\item[(b).] The parameters $\alpha_{k}$ and $\sigma_{k}$ determine the geometry and the possible choice of $y_k$. For each iteration $k$, $\theta(x_{k})$ and $z_k$ with $\psi(x_{k},z_{k})=\theta(x_{k})$ is prior known before utilizing the parameters $\alpha_{k}$ and $\sigma_{k}$ to generate $y_k$. In Algorithm \ref{algorithm1}, $\{\alpha_{k}\}_{k=0}^{\infty}$ and $\{\sigma_{k}\}_{k=0}^{\infty}$ can be taken as not per-determined sequences and thus, for each $k$, the values of $\sigma_{k}$ ($\sigma_{\min}\leq\sigma_{k}\leq 1$) and $\alpha_{k}$ ($\alpha_{\min}\leq\alpha_{k}$) can be adjusted based on the obtained value of $\theta(x_{k})$ and the point $z_k$  so that $y_k$ and $\theta(x_{k})$ satisfy a certain property (may not be any property), for e.g., to satisfy  $\sigma_{k}\theta(x_{k})\geq -\alpha_{k}$ or $\sigma_{k}\theta(x_{k})\leq -\alpha_{k}$. 
\item[(c).] The stopping criteria of {\bf STEP 1} is equivalent to the Pareto stationarity, from the definition of $\theta$ and Lemma \ref{LmS3-02} $(b)$ (see details in Proposition stated next to this). 
\item[(d).] In {\bf STEP 2}, if $d_{k}$ is an $\varepsilon_{k}$-normalized descent direction at $x_{k}$, we take it as the search direction. Otherwise, we construct an $\varepsilon_{k}$-normalized descent direction as shown in Lemma \ref{lemS2-02aa}. Finally, in {\bf STEP 3} we have the line search, as discussed in Remark~\ref{linesearch}.
\item[(e).] The adaptive step size that results from {\bf STEP 2} is \[\xi_{k}=
\begin{cases}
1, \hspace{12mm}\mbox{if}\hspace{1.5mm} \psi(x_{k},y_{k})\leq -\varepsilon_{k}\|y_{k}-x_{k}\|^{2}, \\ 
\frac{t_{k}}{\varepsilon_{k}\|y_{k}-x_{k}\|^{2}}, \hspace{4mm}\mbox{otherwise},
\end{cases}
\]	
where the descent direction $d_{k}$ at $x_{k}$ is given by
$d_{k}=\xi_{k}(y_{k}-x_{k})$. Hence,
\begin{equation}
    \label{eq:iterate}
    x_{k+1}=x_{k}+\lambda_{k}\xi_{k}(y_{k}-x_{k}),
\end{equation}
which means that the iterate is updated with direction $y_{k}-x_{k}$ and step size $\lambda_k \xi_{k}$. \vspace{2mm} 
  \end{description}
  \end{remark}
\begin{proposition}
 \label{LmS-04}
Let $\{x_{k}\}_{k=0}^{\infty}$ and $\{y_{k}\}_{k=0}^{\infty}$ be sequences generated by Algorithm \ref{algorithm1}. Then, the following statements hold.	\begin{description}
		\item[(a).] 
		$\psi(x_{k},y_{k})\leq0$ for all $k$.
		\item[(b).]  $x_{k}$ is Pareto stationary of~(\ref{prob}) if and if 	$\psi(x_{k},y_{k})=0.$ 
	\end{description}  
\end{proposition} 
\begin{proof} 
	Since $\theta(x)\leq 0$ for all $x\in \Omega$ from Lemma \ref{LmS3-02} $(b)$ and the restriction that $\alpha_{k}>0$ for all $k$, we have $
	\psi(x_{k},y_{k})\leq\max\big\{\sigma_{k}\theta(x_{k}),-\alpha_{k}\big\}\leq 0$ for all $k$.	
	This proves $(a).$\\ 	Suppose $x_{k}$ is a Pareto stationary point of (\ref{prob}). Lemma \ref{LmS3-02} $(b)$ implies $\theta(x_{k})=0$, and thus
	$
	0=\theta(x_{k})=\min\limits_{y\in\Omega}\psi(x_{k},y)\leq\psi(x_{k},z)$ for all  $z\in\Omega$.
	Taking $z=y_{k}$, we have $
	0\leq\psi(x_{k},y_{k})$, and combining it with $(a)$ gives
 $\psi(x_{k},y_{k})=0.
	$\\
	Suppose $\psi(x_{k},y_{k})=0.$ From Lemma \ref{LmS3-02} $(b)$, it therefore suffices to show that $\theta(x_{k})\geq 0$.
	Now as a result of  $0=\psi(x_{k},y_{k})$, $0<\sigma_{k}\leq 1$, and $-\alpha_{k}<0$, we find
	\begin{eqnarray}
 \label{theta0}
	0=\psi(x_{k},y_{k})\leq\max\big\{\sigma_{k}\theta(x_{k}),-\alpha_{k}\big\}=\sigma_{k}\theta(x_{k}).
	\end{eqnarray}	
	Since $\sigma_{k}>0$, \eqref{theta0} implies $0\leq \theta(x_{k})$. Hence, in view of Lemma \ref{LmS3-02} $(b)$, we have $\theta(x_{k})=0$, which implies that $x_{k}$ is Pareto stationary of (\ref{prob}). \qed
\end{proof}
From the above proposition, we also conclude that $\{x_k\} \in \Omega$, that is, the algorithm generates feasible iterates with the form~\eqref{eq:iterate}. In fact, recall that either $\xi_k = 1$ or $\xi_k = |\psi(x_{k},y_{k})| / (\varepsilon_k \|y_{k}-x_{k}\|^{2})$. In the latter case, since $|\psi(x_{k},y_{k})| = - \psi(x_{k},y_{k}) < \varepsilon_{k}\|y_{k}-x_{k}\|^{2}$, we get $\xi_k < 1$. Since we also have $\lambda_k \in (0,1]$, the convexity of $\Omega$ shows that $x_{k+1} = (1 - \lambda_k \xi_k) x_k + \lambda_k \xi_k y_k \in \Omega$ when $x_k \in \Omega$. 
Moreover, Proposition \ref{LmS-04} $(b)$ allows us to see that $\psi(x_{k},y_{k})=0$ is a sufficient stopping rule for obtaining the required solution point of the multiobjective optimization problem~(\ref{prob}). Hence, all the next discussions are considered under the nontrivial case of the proposed algorithm where $\psi(x_{k},y_{k})\neq0$ for all $k$.

\begin{lemma} \label{serakeb}
Let $\{d_{k}\}_{k=0}^{\infty}$ and $\{\varepsilon_{k}\}_{k=0}^{\infty}$ be the sequences generated by Algorithm~\ref{algorithm1}. Then, the following statements hold.	\begin{description}
		\item[(a).]  $d_{k}\neq 0$ and $d_{k}$ is an $\varepsilon_{k}$-normalized descent direction of $\mathcal{V}$ at $x_{k}$. 
		
		\item[(b).]
		$\varepsilon_{k}\leq \varepsilon_{k+1}\leq\bar{\varepsilon}$ for all $k$, where $\bar{\varepsilon}:=\frac{L}{2(1-2\beta)}$.
	\end{description} 	
\end{lemma}
\begin{proof} Let us prove $(a)$. The fact that
	$\psi(x_{k},y_{k})\neq0$ implies  $\psi(x_{k},y_{k})<0$ from Proposition~\ref{LmS-04}. Hence, from~\eqref{S3-001} and Step~2 of the algorithm, $y_{k}-x_{k}\neq0$ and $(t_{k}(y_{k}-x_{k}))/(\varepsilon_{k}\|y_{k}-x_{k}\|^{2})\neq0$ which show that $d_{k}\neq0$. Now, recall that if $y_{k}-x_{k}$ is an $\varepsilon_{k}$-normalized descent direction of $\mathcal{V}$ at $x_{k}$, then $d_{k}=y_{k}-x_{k}$. Otherwise, $d_{k}=(|\psi(x_{k},y_{k})|(y_{k}-x_{k}))/(\varepsilon_{k}\|y_{k}-x_{k}\|^{2})$. Hence, by Lemma \ref{lemS2-02} and~\eqref{S3-001} we see that $d_{k}$ is an $\varepsilon_{k}$-normalized descent direction of $\mathcal{V}$ at $x_{k}$, i.e., $\psi(x_{k},x_{k}+d_{k})+\varepsilon_{k}\|d_{k}\|^{2}\leq 0$. \\
    \indent To prove $(b)$, we first note that $\varepsilon_{k}\leq\varepsilon_{k}(1-2\beta)^{1-l_{k}}=\varepsilon_{k+1}$ for all $k$.
	Now, suppose that $2\varepsilon_{k}<L$ for all $k$. From Lemma \ref{lemS2-04}, we see that $(1-2\beta)^{l_{k}}>(2\varepsilon_{k}(1-2\beta)^{2})/L$. Hence, $\varepsilon_{k+1}=\varepsilon_{k}(1-2\beta)^{1-l_{k}}<L/(2(1-2\beta))$ for all $k$. Assume that $2\varepsilon_{k}\geq L$ for some~$k$. Thus from Remark \ref{linesearch}, we have $l_{k}=1$. Hence, $\varepsilon_{k+1}=\varepsilon_{k}(1-2\beta)^{1-l_{k}}=\varepsilon_{k}$ implying that $L\leq 2\varepsilon_{k}=2\varepsilon_{k'}$ for all $k'=k,k+1,k+2,\ldots$. Obviously, $k\geq1$. Let $q$ be the smallest positive integer with the property $2\varepsilon_{q}\geq L$. Thus, $2\varepsilon_{k}<L$ for all $k=1,\ldots,q-1$ and hence  $\varepsilon_{k}<L/(2(1-2\beta))$ for all $k=1,\ldots,q-1$. 
	Due to $2\varepsilon_{q-1}<L$, it follows that $(1-2\beta)^{l_{q-1}}>(2(1-2\beta)^{2}\varepsilon_{q-1})/L$, and hence
	$\varepsilon_{q}=\varepsilon_{q-1}(1-2\beta)^{1-l_{q-1}}<L/(2(1-2\beta))$. Meanwhile, we have $\varepsilon_{k}=\varepsilon_{q}<L/(2(1-2\beta))$ for all $k=q,q+1,q+2,\ldots$. 
	Therefore, this concludes the proof. \qed
\end{proof}
Using Lemma \ref{lemS2-03}, the line search technique of the algorithm is well-defined, in the sense that, the line search procedure in {\bf STEP 3} is always finite, i.e., for each $k\in \mathbb{N}$, there is $l_{k}\in \mathbb{N}$ satisfying (\ref{aaa22}).
\begin{lemma} \label{LmS-044444} Let $\{x_{k}\}_{k=0}^{\infty}$, $\{y_{k}\}_{k=0}^{\infty}$, $\{d_{k}\}_{k=0}^{\infty}$, $\{t_{k}\}_{k=0}^{\infty}$, and $\{\varepsilon_{k}\}_{k=0}^{\infty}$ be the sequences generated by Algorithm \ref{algorithm1}. Then, the
following statements hold.
\begin{description}
\item[(a).]		$\psi(x_{k},x_{k}+d_{k})=\psi(x_{k},y_{k})$ if $y_{k}-x_{k}$ is an $\varepsilon_{k}$-normalized descent direction of $\mathcal{V}$ at $x_{k}$, and $\psi(x_{k},x_{k}+d_{k})\leq\frac{t_{k}}{\varepsilon_{k}\|y_{k}-x_{k}\|^{2}}\psi(x_{k},y_{k})$ otherwise. \item[(b).]	$\psi(x_{k},x_{k}+d_{k})\leq 0$ for all $k$.
  \end{description}
\end{lemma}
\begin{proof} 
	By definition of $d_{k}$ in the algorithm, $d_{k}=y_{k}-x_{k}$ if $y_{k}-x_{k}$ is $\varepsilon_{k}$-normalized descent direction of $\mathcal{V}$ at $x_{k}$. Otherwise, i.e., if $\psi(x_{k},y_{k})> -\varepsilon_{k}\|y_{k}-x_{k}\|^{2}$, we have $d_{k}=\frac{t_{k}(y_{k}-x_{k})}{\varepsilon_{k}\|y_{k}-x_{k}\|^{2}}$ where $t_{k}=\big|\psi(x_{k},y_{k})\big|$. Thus, since $\psi(x_{k},y_{k})<0$ from Proposition \ref{LmS-04}, we get
	$0<t_{k}=\big|\psi(x_{k},y_{k})\big|=-\psi(x_{k},y_{k})<\varepsilon_{k}\|y_{k}-x_{k}\|^{2}$ implying $0<\frac{t_{k}}{\varepsilon_{k}\|y_{k}-x_{k}\|^{2}}\leq1$. Hence, this together with Lemma \ref{LmS3-02} $(d)$ gives $\psi(x_{k},x_{k}+d_{k})\leq\frac{t_{k}}{\varepsilon_{k}\|y_{k}-x_{k}\|^{2}}\psi(x_{k},y_{k})$.

The fact that $\psi(x_{k},y_{k})\leq 0$ for all $k$ (from Proposition \ref{LmS-04}) the result in $(a)$ above gives $\psi(x_{k},x_{k}+d_{k})\leq 0$ for all $k$. \qed
\end{proof}		
\begin{lemma}\label{LmS-05} Define
\begin{equation}  
  \label{eq:gamma}
  \gamma:=\min\left\{\frac{2\varepsilon_{0}(1-2\beta)^{2}}{L},1-2\beta\right\},
\end{equation}
and let $\{x_{k}\}_{k=0}^{\infty}$ and $\{d_{k}\}_{k=0}^{\infty}$ be the sequences generated by Algorithm \ref{algorithm1}. Then, the following statements hold.
	\begin{description}
		\item[(a).]	$\mathcal{V}(x_{k+1})-\mathcal{V}(x_{k})\preceq\gamma\beta \psi(x_{k},x_{k}+d_{k})e_{m}$ for all $k$.
		\item[(b).] $\mathcal{V}(x_{k+1})-\mathcal{V}(x_{k})\preceq-\gamma\beta \varepsilon_{0}\|d_{k}\|^{2}e_{m}$ for all $k$.
		\item[(c).] $\{\mathcal{V}(x_{k})\}_{k=0}^{\infty}$ is a monotonically decreasing sequence.	
	\end{description}	
\end{lemma}
\begin{proof}
	Let $L>2\varepsilon_{k}$. The definition of $x_{k+1}$ and the Armjio-type line search gives
	\begin{eqnarray}\label{ellina1}
	\mathcal{V}(x_{k+1})-\mathcal{V}(x_{k})\nonumber&=&\mathcal{V}(x_{k}+(1-2\beta)^{l_{k}} d_{k})-\mathcal{V}(x_{k})\nonumber\\&\preceq&(1-2\beta)^{l_{k}}\beta (J\mathcal{F}(x_{k})d_{k}+\mathcal{G}(x_{k}+d_{k})-\mathcal{G}(x_{k})-\varepsilon_{k}\|d_{k}\|^{2}e_{m})\nonumber\\&\preceq&(1-2\beta)^{l_{k}}\beta (\psi(x_{k},x_{k}+d_{k})-\varepsilon_{k}\|d_{k}\|^{2})e_{m}\nonumber\\&\preceq&\frac{2\varepsilon_{k}\beta(1-2\beta)^{2}}{L} (\psi(x_{k},x_{k}+d_{k})-\varepsilon_{k}\|d_{k}\|^{2})e_{m}\nonumber\\&\preceq&\frac{2\varepsilon_{0}\beta(1-2\beta)^{2}}{L} (\psi(x_{k},x_{k}+d_{k})-\varepsilon_{k}\|d_{k}\|^{2})e_{m},
	\end{eqnarray}
    where the second inequality holds from~\eqref{S3-001}, the third one follows from 
    Lemma~\ref{lemS2-04}, i.e., $0<(2\varepsilon_{k}(1-2\beta)^{2})/L<(1-2\beta)^{l_{k}},$
	and the fact that $\psi(x_{k},x_{k}+d_{k})-\varepsilon_{k}\|d_{k}\|^{2} \le 0$ (by Lemma~\ref{LmS-044444} $(b)$ and noting $\varepsilon_{k}\|d_{k}\|^{2}\geq 0$ for all $k$), and the last inequality holds because $\varepsilon_{0}\leq\varepsilon_{k}$ for all $k$.\\
	\indent Now, let $L\leq2\varepsilon_{k}$. Thus, from Remark~\ref{linesearch}, $l_{k}=1$. Similarly to~\eqref{ellina1}, we have  
	\begin{equation}\label{ellina2}
	\mathcal{V}(x_{k+1})-\mathcal{V}(x_{k}) \preceq (1-2\beta)\beta \big(\psi(x_{k},x_{k}+d_{k})-\varepsilon_{k}\|d_{k}\|^{2}\big)e_{m}.
	\end{equation} 
	Therefore, since $\psi(x_{k},x_{k}+d_{k})\leq 0$ and $\varepsilon_{k}\|d_{k}\|^{2}\geq 0$ for all $k$, the results $(a)$ and $(b)$ holds straight from (\ref{ellina1}) and (\ref{ellina2}). Finally, from $(a)$ (alternatively $(b)$), we see that $\mathcal{V}(x_{k+1})-\mathcal{V}(x_{k})\preceq 0$ for all $k$, and thus $\{\mathcal{V}(x_{k})\}_{k=0}^{\infty}$ is monotonically decreasing.	\qed
\end{proof}	

Note that from the definition of $d_{k}$, if $\psi(x_{k},y_{k})\leq -\varepsilon_{k}\|y_{k}-x_{k}\|^{2}$, then $\|d_{k}\|=\|y_{k}-x_{k}\|$.  Otherwise, $\|d_{k}\|=\big\|\frac{t_{k}(y_{k}-x_{k})}{\varepsilon_{k}\|y_{k}-x_{k}\|^{2}}\big\|=\frac{t_{k}}{\varepsilon_{k}\|y_{k}-x_{k}\|}\leq\frac{\varepsilon_{k}\|y_{k}-x_{k}\|^{2}}{\varepsilon_{k}\|y_{k}-x_{k}\|}=\|y_{k}-x_{k}\|.$ Therefore, 
\begin{eqnarray}\label{nani29}
\|d_{k}\|\leq\|y_{k}-x_{k}\| \mbox{ for all } k.
\end{eqnarray}
\begin{remark}
\label{remark3333}
Since $\Omega$ is compact and $y_k,x_k\in\Omega$ for all $k$, the sequences $\{x_{k}\}_{k=0}^{\infty}$, $\{y_{k}\}_{k=0}^{\infty}$, and $\{y_{k}-x_{k}\}_{k=0}^{\infty}$ are bounded. Therefore, since $\nabla f_{i}$ is $L_{i}$-Lipschitz continuous for all
	$i=1,\ldots,m$ and $\{x_{k}\}_{k=0}^{\infty}$ is bounded, $\{\nabla f_{i}(x_{k})\}_{k=0}^{\infty}$ is bounded for all
	$i=1,\ldots,m$.  From the boundedness of $\{y_{k}-x_{k}\}_{k=0}^{\infty}$ and $\{\nabla f_{i}(x_{k})\}_{k=0}^{\infty}$, there exist positive real numbers $\delta$ and $\rho_{i}$ ($i=1,\ldots,m$) such that
\begin{eqnarray}\label{eq121212}
\|y_{k}-x_{k}\|\leq\delta \hspace{2mm}\mbox{ and }\hspace{2mm} \|\nabla f_{i}(x_{k})\|\leq\rho_{i}
\end{eqnarray}
  for all $k$.
  \end{remark}
\label{remark3333}We then use the following definitions: \begin{eqnarray}\label{eq232323}\rho=\max\limits_{i=1,\ldots,m}\rho_{i} \hspace{2mm}\mbox{ and }\hspace{2mm} \tau=\frac{\gamma}{\delta(\rho+ L_{G})},\end{eqnarray}
  where $L_G$ and $\gamma$ are defined in~\eqref{eq:Lipschitz} and~\eqref{eq:gamma}, respectively.

\begin{lemma}\label{LmS-06} Let $\{x_{k}\}_{k=0}^{\infty}$,and  $\{d_{k}\}_{k=0}^{\infty}$ be the sequences generated by Algorithm \ref{algorithm1}.
	Then, the following inequalities hold. 
	\begin{description}
		\item[(a).]	$-\tau \psi(x_{k},x_{k}+d_{k})\leq(1-2\beta)^{l_{k}}$ for all $k$.
		
		\item[(b).]
		$
		\tau\varepsilon_{0}\|d_{k}\|^{2}\leq(1-2\beta)^{l_{k}}$ for all $k$.
		
	\end{description}	
\end{lemma}
\begin{proof} If $L\leq2\varepsilon_{k}$, then $l_{k}=1$ from Remark~\ref{linesearch} and hence we have  $\gamma\leq  1-2\beta=(1-2\beta)^{l_{k}}.$ If $L>2\varepsilon_{k}$, then by Lemma \ref{lemS2-04} and the fact that $\varepsilon_0 \le \varepsilon_k$ for all $k$, we get  $\gamma\leq  \frac{2\varepsilon_{0}(1-2\beta)^{2}}{L}\leq\frac{2\varepsilon_{k}(1-2\beta)^{2}}{L}<(1-2\beta)^{l_{k}}.$
	Thus, we have $$0<\gamma\leq(1-2\beta)^{l_{k}}.$$ This, together with Lemma \ref{LmS-05} $(a)$ and $\psi(x_{k},x_{k}+d_{k})\leq 0$ (see Lemma~\ref{LmS-044444} $(b)$) gives
	\begin{eqnarray}\label{S3-003}
	0\leq -\gamma\beta \psi(x_{k},x_{k}+d_{k})\leq&-\beta(1-2\beta)^{l_{k}} \psi(x_{k},x_{k}+d_{k}).
	\end{eqnarray}
	But for all $r=1,\ldots,m$, we have
	\begin{eqnarray}\label{S3-004}
	0\nonumber&\leq&-\psi(x_{k},x_{k}+d_{k})\nonumber\\&=& 
	-\max\limits_{i=1,\ldots,m}\big\{\langle \nabla f_{i}(x_{k}),d_{k} \rangle+g_{i}(x_{k}+d_{k})-g_{i}(x_{k})\big\}\nonumber\\&\leq&  -(\langle \nabla f_{r}(x_{k}),d_{k} \rangle+g_{r}(x_{k}+d_{k})-g_{r}(x_{k}))\nonumber\\&=&  \langle \nabla f_{r}(x_{k}),-d_{k} \rangle-g_{r}(x_{k}+d_{k})+g_{r}(x_{k})\nonumber\nonumber\\&\leq& \|\nabla f_{r}(x_{k})\|\|d_{k}\|+ |g_{r}(x_{k}+d_{k})-g_{r}(x_{k})|\nonumber\\&\leq&  (\|\nabla f_{r}(x_{k})\|+ L_{g_{r}})\|d_{k}\|,
	\end{eqnarray}
    where the third inequality holds from Cauchy-Schwarz and the last one follows from the Lipschitz continuity of~$g_r$.
	Using the boundedness of the sequences $\{y_{k}\}_{k=0}^{\infty}$ and $\{\nabla f_{i}(x_{k})\}_{k=0}^{\infty}$ for all $i=1,\ldots,m$ and the inequalities (\ref{nani29}), (\ref{eq121212}), (\ref{eq232323}), (\ref{S3-003}) and (\ref{S3-004}), we get
	\begin{eqnarray}
	0\leq-\gamma\beta \psi(x_{k},x_{k}+d_{k})\nonumber&\leq&\beta(1-2\beta)^{l_{k}}  (\|\nabla f_{r}(x_{k})\|+ L_{g_{r}})\|d_{k}\|\nonumber\\&\leq&\beta(1-2\beta)^{l_{k}}  (\rho+ L_{G})\delta.\nonumber
	\end{eqnarray}
	Therefore, for the value of $\tau$ in (\ref{eq232323}) the last relation gives
$0\leq	-\tau \psi(x_{k},x_{k}+d_{k})\nonumber\leq (1-2\beta)^{l_{k}}$, that is $(a)$ is satisfied. Moreover, since $d_k$ is an $\varepsilon_k$-normalized descent direction at $x_k$ from Lemma~\ref{serakeb}, and $\varepsilon_0 \le \varepsilon_k$, we have $\tau\varepsilon_{0}\|d_{k}\|^{2}\leq	\tau\varepsilon_{k}\|d_{k}\|^{2}\leq-\tau \psi(x_{k},x_{k}+d_{k})$, which together with $(a)$ gives the inequality $(b)$. \qed
\end{proof}	
\begin{lemma} \label{LmS-07}
Let $\{x_{k}\}_{k=0}^{\infty}$ and $\{d_{k}\}_{k=0}^{\infty}$ be the sequences generated by Algorithm \ref{algorithm1}. Then, the following inequalities hold.
	\begin{description}	
\item[(a).]
$\mathcal{V}(x_{k+1})-\mathcal{V}(x_{k})\nonumber\preceq-\tau\beta \big[\psi(x_{k},x_{k}+d_{k})\big]^{2}e_{m}$ for all $k$.		\item[(b).]
		$\mathcal{V}(x_{k+1})-\mathcal{V}(x_{k})\nonumber\preceq-\tau\beta \varepsilon_{0}^{2}\|d_{k}\|^{4}e_{m}$ for all $k$.			
	\end{description}	
\end{lemma}
\begin{proof} Since 
	$-\tau \psi(x_{k},x_{k}+d_{k})\leq(1-2\beta)^{l_{k}}$ from Lemma \ref{LmS-06} $(a)$, $\psi(x_{k},x_{k}+d_{k})\leq 0$ from Lemma~\ref{serakeb}, $\varepsilon_{k}\|d_{k}\|^{2}\geq 0$, and $\mathcal{V}(x_{k+1})-\mathcal{V}(x_{k})\preceq(1-2\beta)^{l_{k}}\beta (\psi(x_{k},x_{k}+d_{k})-\varepsilon_{k}\|d_{k}\|^{2})e_{m}$ from the Armijo-type condition, we have
		\begin{eqnarray}
	\mathcal{V}(x_{k+1})-\mathcal{V}(x_{k})\preceq(1-2\beta)^{l_{k}}\beta \psi(x_{k},x_{k}+d_{k})e_{m}\preceq-\beta\tau \big[\psi(x_{k},x_{k}+d_{k})\big\}\big]^{2}e_{m}.\nonumber
	\end{eqnarray}
    Similarly, recalling that $\varepsilon_0 \le \varepsilon_k$, we have 
 \begin{eqnarray}
 \mathcal{V}(x_{k+1})-\mathcal{V}(x_{k})\preceq-(1-2\beta)^{l_{k}}\beta \varepsilon_{k}\|d_{k}\|^{2}e_{m}\nonumber&\preceq&-(1-2\beta)^{l_{k}}\beta \varepsilon_{0}\|d_{k}\|^{2}e_{m}\nonumber\\&\preceq&-\tau\beta \varepsilon_{0}^{2}\|d_{k}\|^{4}e_{m},\nonumber
 \end{eqnarray}
	where the last inequality holds from Lemma \ref{LmS-06} $(b)$. \qed
\end{proof}	
\begin{lemma}\label{LmS-06666666}
Define
\begin{equation}
    \label{eq:varpi}
    \varpi:=\min\Big\{\tau,\frac{\gamma}{\bar{\varepsilon}\delta^{2}}\Big\},
\end{equation}
where $\tau$ and $\bar{\varepsilon}$ are given in (\ref{eq232323}) and Lemma~\ref{serakeb}~$(b)$, respectively. Let $\{x_{k}\}_{k=0}^{\infty}$ and $\{y_{k}\}_{k=0}^{\infty}$ be the sequences generated by Algorithm \ref{algorithm1}. Then, for all $k$ we have
$$\mathcal{V}(x_{k+1})-\mathcal{V}(x_{k})\preceq-\varpi\beta \big[\psi(x_{k},y_{k})\big]^{2}e_{m}.$$ 
\end{lemma}
\begin{proof} If $\psi(x_{k},y_{k})\leq -\varepsilon_{k}\|y_{k}-x_{k}\|^{2}$, then $d_{k}=y_{k}-x_{k}$, and hence, using Lemma \ref{LmS-07} $(a)$, we have 	$\mathcal{V}(x_{k+1})-\mathcal{V}(x_{k})\nonumber\preceq-\tau\beta \big[\psi(x_{k},y_{k})\big]^{2}e_{m}$. If $\psi(x_{k},y_{k})> -\varepsilon_{k}\|y_{k}-x_{k}\|^{2}$, then by definition, $d_{k}=\frac{t_{k}(y_{k}-x_{k})}{\varepsilon_{k}\|y_{k}-x_{k}\|^{2}}$, where $t_{k}=\big|\psi(x_{k},y_{k})\big|=-\psi(x_{k},y_{k}).$ Thus, Lemma \ref{LmS-044444} $(a)$ and Lemma \ref{LmS-05} $(a)$  gives
	\begin{eqnarray}
	\mathcal{V}(x_{k+1})-\mathcal{V}(x_{k})\nonumber&\preceq&\gamma\beta \psi(x_{k},x_{k}+d_{k})e_{m}\nonumber\\&\preceq&\frac{\gamma\beta t_{k}}{\varepsilon_{k}\|y_{k}-x_{k}\|^{2}} \psi(x_{k},y_{k})e_{m}=\frac{\gamma\beta \big|\psi(x_{k},y_{k})\big|}{\varepsilon_{k}\|y_{k}-x_{k}\|^{2}} \psi(x_{k},y_{k})e_{m}\nonumber\\&=&\frac{-\gamma\beta}{\varepsilon_{k}\|y_{k}-x_{k}\|^{2}} \big[\psi(x_{k},y_{k})\big]^{2}e_{m}\preceq\frac{-\gamma\beta}{\bar{\varepsilon}\delta^{2}} \big[\psi(x_{k},y_{k})\big]^{2}e_{m},\nonumber
	\end{eqnarray} 
where the last inequality holds from (\ref{eq121212}) and Lemma \ref{serakeb} $(b)$.	 Therefore, from the two cases we obtain 
	$\mathcal{V}(x_{k+1})-\mathcal{V}(x_{k})\nonumber\preceq-\min\big\{\tau,\frac{\gamma}{\bar{\varepsilon}\delta^{2}}\big\}\beta \big[\psi(x_{k},y_{k})\big]^{2}e_{m}$ for all $k$. \qed
\end{proof}	

\begin{lemma}\label{LmS-08}Let $\{x_{k}\}_{k=0}^{\infty}$, $\{y_{k}\}_{k=0}^{\infty}$, and $\{d_{k}\}_{k=0}^{\infty}$ be the sequences generated by Algorithm \ref{algorithm1}. Then, it holds that
\begin{enumerate}
    \item[(a).] $\lim\limits_{k\rightarrow\infty}(\mathcal{V}(x_{k})-\mathcal{V}(x_{k+1}))=0.$
     \item[(b).] $\lim\limits_{k\rightarrow\infty}\|d_{k}\|=\lim\limits_{k\rightarrow\infty}\psi(x_{k},x_{k}+d_{k})=\lim\limits_{k\rightarrow\infty}\psi(x_{k},y_{k})=0.$
\end{enumerate}
\end{lemma}	
\begin{proof}  Since $\{x_{k}\}_{k=0}^{\infty}$ is bounded (see Remark \ref{remark3333}), there exists a subsequence $\{x_{k_{p}}\}$ of $\{x_{k}\}_{k=0}^{\infty}$ such that $x_{k_{p}}\rightarrow x^{*}$ as $p\rightarrow \infty$. From $\mathcal{F}$ is continuous and $x_{k_{p}}\rightarrow x^{*}$, we have $\mathcal{F}(x_{k_{p}})\rightarrow \mathcal{F}(x^{*})$ as $p\rightarrow \infty$. Moreover, from $\mathcal{V}=\mathcal{F}+\mathcal{G}$ we have
 \begin{align}
 \label{corr11}
     \|\mathcal{V}(x_{k_{p}})-\mathcal{V}(x^{*})\|\nonumber&\leq \|\mathcal{F}(x_{k_{p}})-\mathcal{F}(x^{*})\|+\|\mathcal{G}(x_{k_{p}})-\mathcal{G}(x^{*})\|\\&\leq \|\mathcal{F}(x_{k_{p}})-\mathcal{F}(x^{*})\|+L_{G}\|x_{k_{p}}-x^{*}\|,
 \end{align}
 where the last inequality is obtained since $\mathcal{G}$ is $L_{G}$-Lipschitz continuous. Hence, \eqref{corr11} in view of $x_{k_{p}}\rightarrow x^{*}$ and $\mathcal{F}(x_{k_{p}})\rightarrow \mathcal{F}(x^{*})$ yields  $\mathcal{V}(x_{k_{p}})\rightarrow \mathcal{V}(x^{*})$. Since the sequence $\{\mathcal{V}(x_{k})\}_{k=0}^{\infty}$ is  monotonically decreasing from Lemma \ref{LmS-05} $(c)$, the whole sequence $\{\mathcal{V}(x_{k})\}_{k=0}^{\infty}$  converges and $\mathcal{V}(x_{k})\rightarrow \mathcal{V}(x^{*})$ as $k\rightarrow \infty$. This implies
	\begin{eqnarray}\label{ll111}
	\lim\limits_{k\rightarrow\infty}(\mathcal{V}(x_{k})-\mathcal{V}(x_{k+1}))=0.
	\end{eqnarray}
	Therefore, (\ref{ll111}) and Lemma \ref{LmS-07} $(a)$ and $(b)$ (alternatively Lemma \ref{LmS-05} $(a)$ and $(b)$) give  $\lim\limits_{k\rightarrow\infty}\|d_{k}\|=\lim\limits_{k\rightarrow\infty}\psi(x_{k},x_{k}+d_{k})=0$. Moreover, $\lim\limits_{k\rightarrow\infty}\psi(x_{k},y_{k})=0$ is immediate from (\ref{ll111}) and Lemma \ref{LmS-06666666}. \qed
\end{proof}
\section{Convergence Analysis}\label{Sec4}
As a consequence of the facts of Section \ref{Sec3}, we now study the convergence rate and iteration complexity of the sequences generated by Algorithm \ref{algorithm1} ({\bf A-GCG}).
The following is the main theorem related to the  convergence with respect to Pareto criticality.
\begin{theorem}\label{Thm-01}
	Every limit point of the sequence $\{x_{k}\}_{k=0}^{\infty}$ generated by Algorithm \ref{algorithm1} is a Pareto critical point of (\ref{prob}).
\end{theorem}
\begin{proof}
	Since the sequence $\{x_{k}\}$ is  bounded, and hence it has at least one limit point. Let $\bar{x}$ be 	the limit point of $\{x_{k}\}$, and $\{x_{k_{q}}\}$ be the subsequence of $\{x_{k}\}$ such that 	$x_{k_{q}}\rightarrow \bar{x} \mbox{ as } q\rightarrow \infty$. Since  
 $g_{i}$ is $L_{g_{i}}$-Lipschitz continuous, we have $g_i(x_{k_{p}})\rightarrow g_i(\bar{x})$. Using Cauchy–Schwarz inequality and \eqref{eq121212}, we have  $$|\langle\nabla f_{i}(x_{k_{q}}),\bar{x}-x_{k_{q}}\rangle|\leq \|\nabla f_{i}(x_{k_{q}})\|\|\bar{x}-x_{k_{q}}\|\leq \rho_{i}\|\bar{x}-x_{k_{q}}\|,$$
 which implies by $x_{k_{q}}\rightarrow \bar{x}$ that  $\langle\nabla f_{i}(x_{k_{q}}),\bar{x}-x_{k_{q}}\rangle\rightarrow 0$ for all $i=1,\ldots,m$. Moreover, the Lipschitz contiunity of $\nabla f_{i}(.)$ gives $\nabla f_{i}(x_{k_{q}})\rightarrow \nabla f_{i}(\bar{x})$, and thus $\langle\nabla f_{i}(x_{k_{q}}),z\rangle\rightarrow\langle\nabla f_{i}(\bar{x}),z\rangle$ for all $z\in\mathbb{R}^n$ and $i=1,\ldots,m$. Now, for each $i=1,\ldots,m$, 
 \begin{align}
   \langle\nabla f_{i}(x_{k_{q}}),y-x_{k_{q}}\rangle=\langle\nabla f_{i}(x_{k_{q}}),y-\bar{x}\rangle+\langle\nabla f_{i}(x_{k_{q}}),\bar{x}-x_{k_{q}}\rangle.\nonumber
 \end{align} Hence, $\langle\nabla f_{i}(x_{k_{q}}),y-x_{k_{q}}\rangle\rightarrow\langle\nabla f_{i}(\bar{x}),y-\bar{x}\rangle$ for all $i=1,\ldots,m$. By $\langle\nabla f_{i}(x_{k_{q}}),y-x_{k_{q}}\rangle\rightarrow\langle\nabla f_{i}(\bar{x}),y-\bar{x}\rangle$,
 $g_i(x_{k_{p}})\rightarrow g_i(\bar{x})$ and  continuity of maximum function we obtain 
 \begin{equation}\label{converg}
\psi(x_{k_{q}},z)\rightarrow\psi(\bar{x},z) \hspace{2mm}\mbox{ as } \hspace{2mm}q\rightarrow \infty. 
 \end{equation}
 From the algorithm and the 	definition of the estimates $\alpha_{\min}$ and $\sigma_{\min}$, we have
	\begin{eqnarray}\label{S3-009}
	\psi(x_{k},y_{k})\nonumber&\leq&\max\big\{\sigma_{k}\theta(x_{k}),-\alpha_{k}\big\}\nonumber\\&\leq&\max\Big\{\sigma_{\min}\theta(x_{k}),-\alpha_{\min}\Big\}\nonumber\\&\leq&\max\big\{\sigma_{\min}\psi(x_{k},z),-\alpha_{\min}\big\}
	\end{eqnarray}
	for all $z\in \Omega$, where we also used the fact that $\theta(x_k) \le 0$ (see Lemma~\ref{LmS3-02} $(b)$). Since $\lim\limits_{k\rightarrow\infty}\psi(x_{k},y_{k})=0$ by Lemma \ref{LmS-08}, it follows from \eqref{converg} and \eqref{S3-009} that	
	\begin{eqnarray}
	0=\lim\limits_{q\rightarrow\infty}\psi(x_{k_{q}},y_{k_{q}})\nonumber&\leq&\lim\limits_{q\rightarrow\infty}\big(\max\big\{\sigma_{\min}\psi(x_{k_{q}},z),-\alpha_{\min}\big\}\big)\nonumber\\&=&\sigma_{\min}\lim\limits_{q\rightarrow\infty}\psi(x_{k_{q}},z)\nonumber\\&=&\sigma_{\min}\psi(\bar{x},z)\nonumber
	\end{eqnarray}	
	which implies that $$\psi(\bar{x},z)=\max\limits_{i=1,\ldots,m}\big[\langle\nabla f_{i}(\bar{x}),z-\bar{x}\rangle +g_{i}(z)-g_{i}(\bar{x})\big]\geq 0$$ for all $z\in   \Omega$. That is $\theta(\bar{x})=\min\limits_{z\in\Omega} \psi(\bar{x},z)\geq 0$, and  therefore, from Lemma \ref{LmS3-02} $(b)$, $\bar{x}$ is a Pareto critical point of (\ref{prob}). \qed 
\end{proof} 

\begin{assumption} \label{A-3}
	For all $k$,
	$\{\alpha_{k}\}_{k=0}^{\infty}$ and $\{\sigma_{k}\}_{k=0}^{\infty}$ are chosen such that	$$\max\big\{\sigma_{k}\theta(x_{k}),-\alpha_{k}\big\}=\sigma_{k}\theta(x_{k}) \mbox{ for all } k.$$	
\end{assumption}
 Assumption \ref{A-3} is manageable as we discussed in Remark \ref{remark2} $(b)$. The next results are considered under this additional assumption.
\begin{lemma}\label{added11}
Let $\{x_{k}\}_{k=0}^{\infty}$ be the sequence generated by Algorithm \ref{algorithm1}. Then, it holds that 
$$\varpi\beta\sigma^2_{\min} \big[\theta(x_{k})\big]^{2}\leq\mathcal{V}_{i}(x_{k})-\mathcal{V}_{i}(x_{k+1}) \hspace{4mm} \mbox{ and }\hspace{4mm} \lim\limits_{k\rightarrow\infty}\theta(x_{k})=0,$$
where $\varpi$ is given in~\eqref{eq:varpi}.
\end{lemma}	
\begin{proof}  Note that under Assumption \ref{A-3}, from the algorithm and Lemma \ref{LmS3-02} $(b)$, we have 
$|\psi(x_{k},y_{k})|\geq-\max\big\{\sigma_{k}\theta(x_{k}),-\alpha_{k}\big\}=\sigma_{k}|\theta(x_{k})|\geq\sigma_{\min}|\theta(x_{k})|$ for all $k$. 
Hence, in view of Lemma \ref{LmS-06666666}, we get
$$\varpi\beta\sigma_{\min}^2 \big[\theta(x_{k})\big]^{2}\leq\varpi\beta[\psi(x_{k},y_{k})]^2\leq\mathcal{V}_{i}(x_{k})-\mathcal{V}_{i}(x_{k+1})$$
for all $k$, which also gives $\lim\limits_{k\rightarrow\infty}\theta(x_{k})=0$ using the fact $	\lim\limits_{k\rightarrow\infty}(\mathcal{V}(x_{k})-\mathcal{V}(x_{k+1}))=0$ from Lemma~\ref{LmS-08} $(a)$. \qed 
\end{proof} 

Observe that under Assumption~\ref{A-3}, the above result can be also used to show Theorem~\ref{Thm-01}, because Lemma~\ref{LmS3-02} holds. We now consider approximate Pareto critical points, as defined below, and show the number of iterations required by the algorithm to find such points. In the subsequent results, we consider the following notations:
\begin{align}
\mathcal{V}_{i^{*}}(x_{0}) := & \max\{\mathcal{V}_{i}(x_{0}):i=1,\ldots,m\}, \nonumber \\
 \mathcal{V}^{\mathrm{inf}} :=  & \min\{\mathcal{V}_{i}^{*}:i=1,\ldots,m\},
    \quad \mbox{ where } \mathcal{V}_{i}^{*}:=\inf\{\mathcal{V}_{i}(x):x\in\Omega\}.
    \label{eq:v_inf} 
\end{align}

\begin{definition} Let $\mu\geq0$. The point $x\in \Omega$ is called $\mu$-approximate Pareto critical point of~(\ref{prob}) if  $\theta(x)\geq -\mu$.
\end{definition}
Note that based on Lemma~\ref{LmS3-02} $(b)$, $x\in \Omega$ is $0$-approximate Pareto critical point of~(\ref{prob}) if and only if $x$ is a Pareto critical point of (\ref{prob}). For each $\mu>0$, in the following (Corollary \ref{remm-7775566}), we will show that within ${\bf\it {O}}(1/\mu^{2})$ iterations, Algorithm \ref{algorithm1} yields the first $\mu$-approximate Pareto critical point $x_{k}$.
\begin{corollary}\label{remm-7775566}
	For $\mu>0$, Algorithm \ref{algorithm1} finds a $\mu$-approximate Pareto critical point of~(\ref{prob}) within ${\bf\it {O}}(1/\mu^{2})$ iterations. 
\end{corollary}
\begin{proof} 
For a nonnegative integer $N$, summing up the inequality of Lemma~\ref{added11} from $k=0$ to $k=N$ and using~\eqref{eq:v_inf}, we obtain
$$
\varpi\beta\sigma^2_{\min} \sum_{k=0}^{N}
\big[\theta(x_{k})\big]^{2}\leq\mathcal{V}_{i}(x_{0})-\mathcal{V}_{i}(x_{N+1})\leq\mathcal{V}_{i^{*}}(x_{0})-\mathcal{V}^{\mathrm{inf}},$$
and therefore, \begin{eqnarray}\label{lesttina}
\min\limits_{k=0,1,\ldots,N}|\theta(x_{k})|^{2}\leq\frac{\mathcal{V}_{i^{*}}(x_{0})-\mathcal{V}^{\mathrm{inf}}}{\varpi\beta\sigma^2_{\min} (N+1)}.
\end{eqnarray}	Since $\lim\limits_{k\rightarrow\infty}\theta(x_{k})=0$ from Lemma \ref{added11}, there exists $n_{\mu}\in\mathbb{N}$ such that $-\theta(x_{k})=|\theta(x_{k})|\leq \mu$ for all $k\geq n_{\mu}$. This implies that the set $\Delta_{\mu}:=\big\{k:\theta(x_{k})<-\mu\}=\big\{k:|\theta(x_{k})|>\mu\}$ is finite. 
	 If $\Delta_{\mu}=\emptyset$,  the result clearly holds. Let $\Delta_{\mu}\neq\emptyset$. Then, there exists $N_{\mu}$ such that
	 $|\theta(x_{k})|>\mu$ 
	 for all $k\in\{0,1,\ldots,N_{\mu}-1\}$  
but $|\theta(x_{N_{\mu}})|\leq\mu$ ($x_{N_{\mu}}$ is a $\mu$-approximate Pareto critical point of (\ref{prob})).  Hence, $\min\limits_{k=0,1,\ldots,(N_{\mu}-1)}|\theta(x_{k})|>\mu$, and applying (\ref{lesttina}) by taking $N=N_{\mu}-1$, we get
	\begin{eqnarray}
	N_{\mu}\mu^{2}<N_{\mu}\Big(\min\limits_{k=0,1,\ldots,(N_{\mu}-1)}|\theta(x_{k})|\Big)^{2}\nonumber&=& N_{\mu}\min\limits_{k=0,1,\ldots,(N_{\mu}-1)}|\theta(x_{k})|^{2}\nonumber\\&\leq&\frac{\mathcal{V}_{i^{*}}(x_{0})-\mathcal{V}^{\mathrm{inf}}}{\varpi\beta\sigma^2_{\min}}.\nonumber\end{eqnarray}	
	Therefore, the algorithm finds $x_{N_{\mu}}$ within $\frac{\mathcal{V}_{i^{*}}(x_{0})-\mathcal{V}^{\mathrm{inf}}}{\varpi\beta\sigma^2_{\min}\mu^{2}}$ iterations. \qed  
\end{proof}	

\begin{theorem}\label{thm22222}  Let $\{x_{k}\}_{k=0}^{\infty}$ and $\{y_{k}\}_{k=0}^{\infty}$ be sequences generated by Algorithm \ref{algorithm1}, and $\psi$, $\theta$, and $u_{0}$ be defined by (\ref{S3-001}),  (\ref{S3-002}), and (\ref{merit}), respectively. Then, the following statements hold.
\begin{description}
	\item[(a).]		$	-\sigma_{\min}u_{0}(x_{k})\geq\sigma_{\min}\theta(x_{k}) \geq\psi(x_{k},y_{k}) \mbox{ for all } k.
	$
	\item[(b).]	There exists $\{q_{k}\}_{k=0}^{\infty}$ such that 	$$u_{0}(x_{k+1})\leq q_{k} u_{0}(x_{k}) \mbox{ for all } k,$$		
where $\frac{1}{2}<q_{k}<1$  for all $k$ and $\lim\limits_{k\rightarrow\infty}q_{k}=1$.
\end{description}	
\end{theorem} 
\begin{proof}
	Note that from the differentiability of $f_i$, we have
	$$	\mathcal{V}_{i}(y)-\mathcal{V}_{i}(x_{k})=f_{i}(y)-f_{i}(x_{k})+g_{i}(y)-g_{i}(x_{k})\geq\langle\nabla f_{i}(x_{k}),y-x_{k}\rangle +g_{i}(y)-g_{i}(x_{k})$$
	for all $k$, $y\in\Omega$, and $i=1,\ldots,m$. This yields
	\begin{eqnarray}	-\sigma_{\min}u_{0}(x_{k})\nonumber&=&\sigma_{\min}\inf_{y\in\Omega}\max_{i=1,\ldots,m}\big(\mathcal{V}_{i}(y)-\mathcal{V}_{i}(x_{k})\big)\nonumber\\&\geq&\sigma_{\min}\inf_{y\in\Omega}\max_{i=1,\ldots,m}\big(\langle\nabla f_{i}(x_{k}),y-x_{k}\rangle +g_{i}(y)-g_{i}(x_{k})\big)	\nonumber\\&=&\sigma_{\min}\min_{y\in\Omega}\psi(x_{k},y)\nonumber\\&=&\sigma_{\min}\theta(x_{k})\geq\sigma_{k}\theta(x_{k}) =\max\big\{\sigma_{k}\theta(x_{k}),-\alpha_{k}\big\}\geq\psi(x_{k},y_{k}),\nonumber
	\end{eqnarray}
	where the second equality holds from~\eqref{S3-001}, the second inequality follows from Lemma~\ref{LmS3-02} $(b)$, the last equality is true from Assumption \ref{A-3}, and the last inequality holds from Step~1 of the algorithm. Thus, we conclude that $(a)$ holds.\\
    \indent Now, let us prove $(b)$. Due to $\psi(x_{k},y_{k})\leq 0$ for all $k$ from Proposition \ref{LmS-04} $(a)$ and  $\frac{t_{k}}{\bar{\varepsilon}\delta^{2}}\leq\frac{t_{k}}{\varepsilon_{k}\|y_{k}-x_{k}\|^{2}}$
	from Lemma \ref{serakeb} and (\ref{eq121212}), it follows from Lemma \ref{LmS-044444} $(a)$ that
	\begin{eqnarray}\label{ellina212121}\psi(x_{k},x_{k}+d_{k})\nonumber&\leq&\min\Big\{1,\frac{t_{k}}{\varepsilon_{k}\|y_{k}-x_{k}\|^{2}}\Big\}\psi(x_{k},y_{k})\\&\leq&\min\Big\{1,\frac{t_{k}}{\bar{\varepsilon}\delta^{2}}\Big\}\psi(x_{k},y_{k})\end{eqnarray}
	for all $k$. Combining \eqref{ellina212121} with the fact $-\sigma_{\min}u_{0}(x_{k})\geq\psi(x_{k},y_{k})$ from  $(a)$, we obtain $\psi(x_{k},x_{k}+d_{k})\leq-\sigma_{\min}\min\{1,\frac{t_{k}}{\bar{\varepsilon}\delta^{2}}\}u_{0}(x_{k})$ for all $k$. Thus, from Lemma \ref{LmS-05} $(a)$, we have $\mathcal{V}_{i}(x_{k+1})\leq\mathcal{V}_{i}(x_{k})-\gamma\beta\sigma_{\min}\min\{1,\frac{t_{k}}{\bar{\varepsilon}\delta^{2}}\} u_{0}(x_{k})$ for all $k$ and $i=1,\ldots,m$. That is, \begin{eqnarray}\label{ellina272727}\mathcal{V}_{i}(x_{k+1})-\mathcal{V}_{i}(y)\leq\mathcal{V}_{i}(x_{k})-\mathcal{V}_{i}(y)-\gamma\beta\sigma_{\min}\min\Big\{1,\frac{t_{k}}{\bar{\varepsilon}\delta^{2}}\Big\} u_{0}(x_{k})\end{eqnarray}
	for all $k$,  $i=1,\ldots,m$, and $y\in\Omega$. Therefore, (\ref{ellina272727}) gives
	\begin{eqnarray}	\nonumber&&\sup_{y\in\Omega}\min_{i=1,\ldots,m}\big(\mathcal{V}_{i}(x_{k+1})-\mathcal{V}_{i}(y))\nonumber\\&&\hspace{20mm}\leq\sup_{y\in\Omega}\min_{i=1,\ldots,m}\big(\mathcal{V}_{i}(x_{k})-\mathcal{V}_{i}(y))-\gamma\beta\sigma_{\min}\min\Big\{1,\frac{t_{k}}{\bar{\varepsilon}\delta^{2}}\Big\} u_{0}(x_{k})\nonumber\end{eqnarray}
	for all $k$, which by definition of $u_{0}$ it means that	  $u_{0}(x_{k+1})\leq q_{k} u_{0}(x_{k})$, where $q_{k}:=1-\gamma\beta\sigma_{\min}\min\{1,\frac{t_{k}}{\bar{\varepsilon}\delta^{2}}\}$. In view of $t_{k}=|\psi(x_{k},y_{k})|\neq0$, $\lim\limits_{k\rightarrow\infty}t_{k}=\lim\limits_{k\rightarrow\infty}|\psi(x_{k},y_{k})|=0$ from Lemma \ref{LmS-08}, the definition $\gamma$ given in~\eqref{eq:gamma}, and the values of $\beta$ and $\sigma_{\min}$ in the algorithm, we can see that $0<\gamma\beta\sigma_{\min}\min\{1,\frac{t_{k}}{\bar{\varepsilon}\delta^{2}}\}< \frac{1}{2}$, that is, $\frac{1}{2}<q_{k}<1$ for all $k$ and $\lim\limits_{k\rightarrow\infty}q_{k}=1$. \qed
\end{proof}

\begin{lemma} \cite[Chapter 2, Lemma 6]{16}\label{lem-yetu}
	If $\xi>0$ and $\{\varphi_{k}\}$ is a sequence of non-negative real numbers with
	$	\xi\varphi_{k}^{2}\leq\varphi_{k}-\varphi_{k+1},$ then 
	$\varphi_{k}\leq \frac{\varphi_{0}}{1+\xi\varphi_{0}k}$ for all  $k$.
\end{lemma}	
In the following theorem, we will show the $\{u_{0}(x_{k})\}$ converges sublinearly to zero with a rate ${ {\bf\it O}}(1/k)$, where $u_{0}$ is the merit function  defined by (\ref{merit}).
\begin{theorem}\label{thm3333} (Sublinear rate of convergence)	Let $\{x_{k}\}_{k=0}^{\infty}$ be the sequence generated by Algorithm \ref{algorithm1} and $u_{0}$ be defined by (\ref{merit}). If there exists  $x^{*}\in\Omega$ such that $\mathcal{V}(x^{*})\preceq\mathcal{V}(x_{k})$ for all $k$, then 
	$$u_{0}(x_{k})<\frac{1}{\varpi\beta\sigma_{\min}^{2} k} \mbox{ for all } k\in\mathbb{N}.$$
\end{theorem}
\begin{proof} 
	From the existence of a point $x^{*}\in\Omega$ with $\mathcal{V}(x^{*})-\mathcal{V}(x_{k})\preceq 0$ for all $k$ (see Lemma~\ref{LmS-05}), we have 
    \begin{eqnarray}	0\geq\sigma_{\min}\max\limits_{i=1,\ldots,m}\big(\mathcal{V}_{i}(x^{*})-\mathcal{V}_{i}(x_{k})\big)\nonumber&\geq&
	\sigma_{\min}\inf\limits_{y\in\Omega}\max\limits_{i=1,\ldots,m}\big(\mathcal{V}_{i}(y)-\mathcal{V}_{i}(x_{k})\big)\nonumber\\&=&-\sigma_{\min}u_{0}(x_{k}).\nonumber \end{eqnarray}
    Hence, by $-\sigma_{\min}u_{0}(x_{k})\geq\psi(x_{k},y_{k})$ from Theorem \ref{thm22222} $(a)$, we obtain 
	\begin{eqnarray}\label{pommy192127}
	0\leq\sigma_{\min}^{2}u_{0}^{2}(x_{k})\leq [\psi(x_{k},y_{k})]^{2}
	\end{eqnarray}
	for all $k$. Lemma \ref{LmS-06666666} and (\ref{pommy192127}) implies
	$$\mathcal{V}_{i}(x_{k+1})-\mathcal{V}_{i}(y)\leq\mathcal{V}_{i}(x_{k})-\mathcal{V}_{i}(y)-\varpi\beta \sigma_{\min}^{2}u_{0}^{2}(x_{k})$$
	for all $k$,  $i=1,\ldots,m$, and $y\in\Omega$, and thus, taking $\sup\limits_{y\in\Omega}\min\limits_{i=1,\ldots,m}$ in both sides and using the definition of $u_{0}$, we have \begin{eqnarray}\label{yetumist21212121}u_{0}(x_{k+1})\leq u_{0}(x_{k})-\varpi\beta \sigma_{\min}^{2}u_{0}^{2}(x_{k})\end{eqnarray} for all $k$. Letting $\varphi_{k}:=u_{0}(x_{k})$, the inequality  (\ref{yetumist21212121}) is rewritten as $	\varpi\beta\sigma_{\min}^{2}\varphi_{k}^{2}\leq\varphi_{k}-\varphi_{k+1}$ for all $k$. Therefore, applying Lemma \ref{lem-yetu}, we get
	$\varphi_{k}\leq\frac{\varphi_{0}}{1+\varpi\beta\sigma_{\min}^{2}\varphi_{0}k}$ for all $k$, and hence, $\varphi_{k}<\frac{1}{\varpi\beta\sigma_{\min}^{2} k}$ for all $k\in\mathbb{N}$. 
	This completes the proof.	\qed
\end{proof}
\section{Numerical Experiments}\label{Sec5}
The purpose of this section is to carry out some
numerical experiments to verify and demonstrate the performance of Algorithm \ref{algorithm1} ({\bf A-GCG}) in comparison to the line search-based form of the generalized conditional gradient
method ({\bf GCG}) in \cite{2},  proximal gradient method ({\bf PG}) in \cite{4aa}, and the conditional gradient method ({\bf CG}) in \cite{1}.
There are several types of problems of the form (\ref{prob}) considered in the literature (theoretical \cite{13,4aa} and practical \cite{18,17} problems, and see the list of test problems in \cite{2,22eliant}). We compare {\bf A-GCG} with {\bf GCG}, {\bf PG}, and {\bf CG} for the test problems in Example \ref{exam1},  \ref{exam2}, and \ref{exam3} in reference to satisfying the Stopping Criteria: \begin{eqnarray}\label{stop1}
\frac{|\theta(x_{k})|}{\max\{\frac{1}{n}\|x_0\|,1\}}\leq\mu.	\end{eqnarray}
This means that the algorithms stop for the smallest $k$ such that the point $x_{k}$ is $\mu \max\{\frac{1}{n}\|x_0\|,1\}$-approximate Pareto optimal of the multiobjective optimization considered. 
\begin{itemize}
\item[$\bullet$] The tables present the approximate number of iterations (\#Iter),  CPU time of execution in seconds (Cput), number of computations of  $\mathcal{V}$ (\#Fc), and number of gradient computations performed  (\#$\nabla$) required to attain \eqref{stop1}. The values \#Iter, Cput, \#Fc, and \#$\nabla$ are the approximated average value of three running outputs.
\item[$\bullet$] The figures show the path $(\mathcal{V}_{1}(x_{k}),\ldots, \mathcal{V}_{m}(x_{k}))$ and the approximate Pareto optimal points obtained by the algorithms using multiple starting points. 
\item[$\bullet$] {\bf A-GCG} is implemented by setting $\varepsilon_{0}=0.01$ and $\beta=0.2$ and selecting the parameters $\alpha_{k}$ and $\sigma_{k}$ such that $0<\alpha_{\min}=10^{2}\leq \alpha_{k}$ and $0<\sigma_{\min}=10^{-2}\leq \sigma_{k}\leq 1$. In {\bf STEP 1} of {\bf A-GCG} we set  $y_{k}=x_{k}+\sigma(z_{k}-x_{k})$ where $z_{k}\in\Omega$ and
$\psi(x_{k},z_{k})=\theta(x_{k})$, and $\sigma$ is randomly chosen number with $\sigma_{\min}\leq\sigma_{k}\leq\sigma\leq 1$.
\item[$\bullet$] The subproblem \eqref{S3-002} can be reformulated as an equivalent problem of finding the points $(\Gamma,y)\in\mathbb{R}\times\mathbb{R}^{n}$ given by
\begin{align}
&\min_{\Gamma,y} \Gamma\label{cons1}\\&  
  \mbox{s.t}\hspace{2mm} \langle\nabla f_{i}(x),y-x\rangle +g_{i}(y)-g_{i}(x)\leq \Gamma, \hspace{3mm} \forall i=1,\ldots,m\label{cons2}\\&\hspace{6mm}
  y\in\Omega.
\end{align} 
\end{itemize} 

All numerical experiments are implemented in MATLAB R2024a and executed on a MacBook Air with 8GB.  We use the linear programming command ({\it linprog}) when applying {\bf A-GCG}, {\bf GCG}, and {\bf CG}, and the quadratic programming ({\it quadprog}) when applying {\bf PG}. 
\begin{example}\label{exam1} 
	Consider the multiobjective optimization  given by
\begin{eqnarray}\label{ex3333}
\nonumber&&\min \: \mathcal{V}(x):=\Big(\frac{1}{n}\|x-a_1\|^{2},\ldots,\frac{1}{n}\|x-a_m\|^{2}\Big)\\&&\hspace{1mm}\mbox{s.t. } \: x\in [-\ell,\ell]^n.
	\end{eqnarray}
 where $a_1,a_2\in\mathbb{R}^n$ and $\ell>0$.
\end{example}
The problem (\ref{ex3333}) can be reformulated as (\ref{prob}), where $\mathcal{V}_{i}(x):=\frac{1}{n}\|x-a_i\|^{2}=f_{i}(x)+g_{i}(x)$ for 
$f_{i}(x)=\frac{1}{n}\|x-a_i\|^{2}$ and $g_{i}(x)=0$  for $i=1,\ldots,
m$. We take $a_1$ and $a_2$ to be a randomly generated vectors in  $[0,1]^n$. 
The comparative performance of {\bf A-GCG}, {\bf GCG}, {\bf PG}, and {\bf CG} in the implementation for (\ref{ex3333}) for different $m$, $\ell$, and $n$ value is shown in Figures \ref{Figure1} and \ref{Figure2}, and Table \ref{Table1}. \vspace{1mm}  
\begin{figure}[h] 
	\centering
	\includegraphics[scale=0.31]{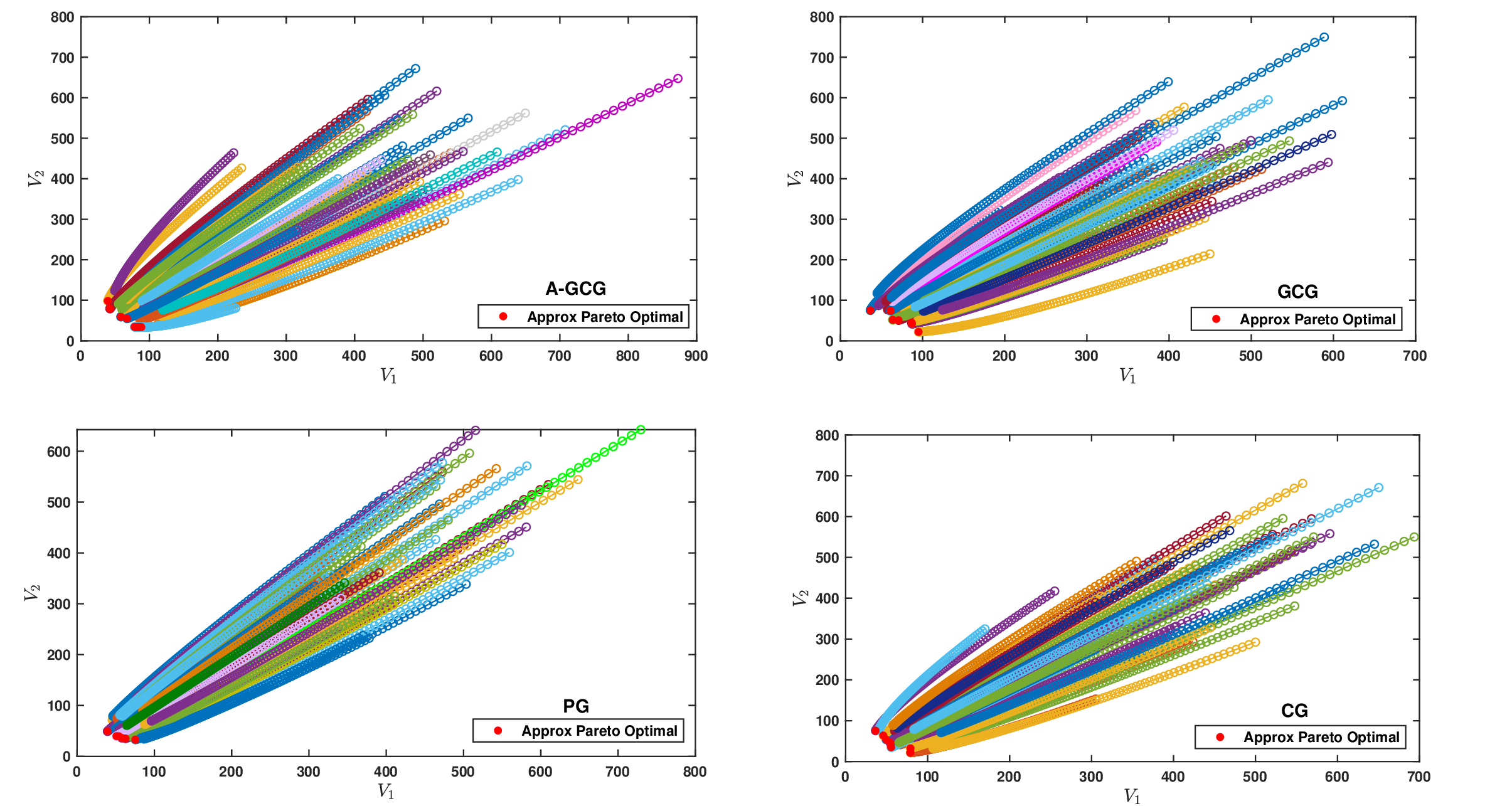}
 \caption{Iterative path and approximate Pareto optimal obtained for Example \ref{exam1} ($m=2$, $\ell=10$, and $n=15$) using 100 different starting points  in $[-10,10]^n$ and $\mu=10^{-3}$.} 
 \label{Figure1}
\end{figure}
\begin{figure}[h] 
	\centering
	\includegraphics[scale=0.32]{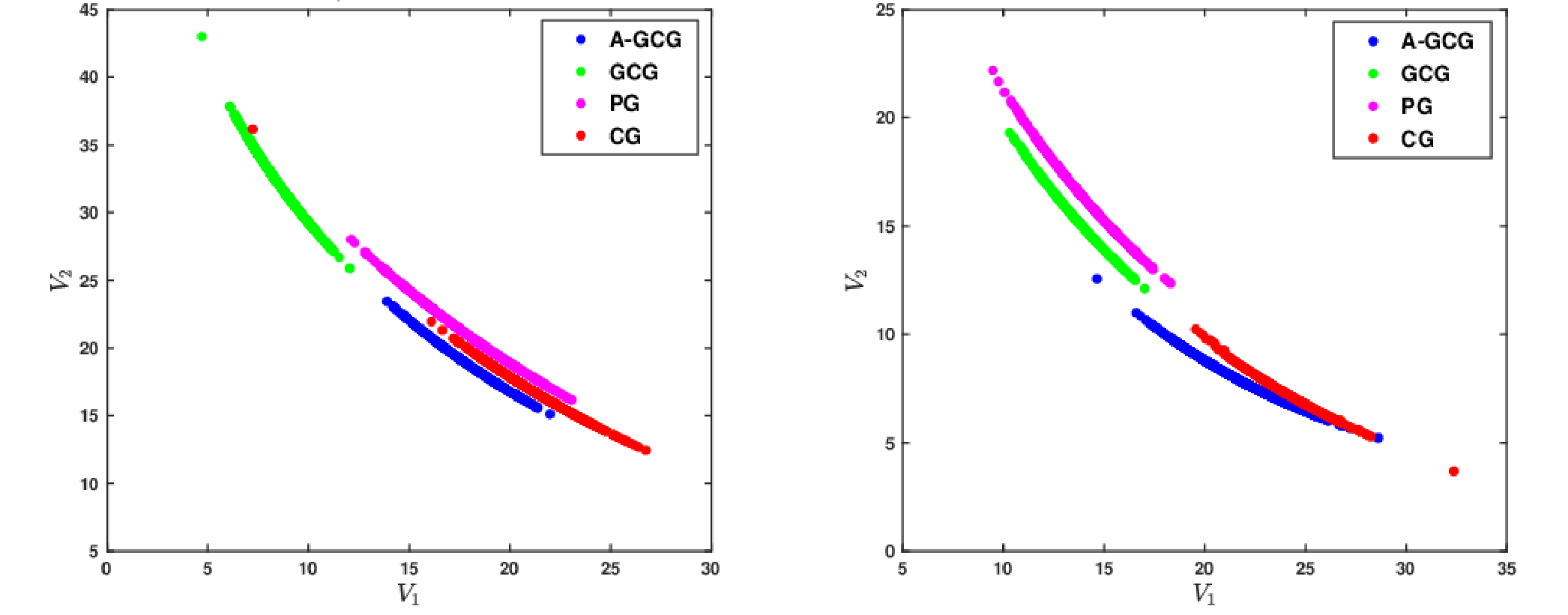} 	
	\caption{Approximate Pareto optimal  obtained for Example \ref{exam1} ($m=2$, $\ell=100$, and $n=20$) using 200 different starting points in $[-100,100]^n$, and $\mu=10^{-3}$ (left side figure) and $\mu=10^{-4}$ (right side figure).}\label{Figure2}
\end{figure}
\begin{table}[]
\centering
	\caption{Example \ref{exam1} for $x_0=2(1,\dots,1)$, $n=25$, and $\mu=10^{-3}$.}\label{Table1}
\begin{tabular}{lllllllllll}
\hline
 & & \multicolumn{4}{l}{$\ell=10$} &  & \multicolumn{4}{l}{$\ell=100$} \\ \cline{3-6} \cline{8-11} 
Solver  & $m$ & $\#$Iter    &  CPU   &  $\#$Fc   &  $\#$$\nabla$  &  &    $\#$Iter.    &  CPU   &  $\#$Fc   &  $\#$$\nabla$    \\ \hline
A-GCG & 2 &   16 & 0.04834   & 172   &  89  &  &   17  &   0.05018 &  178   &  89 \\
GCG &  2 &  27 &  0.04246   &  206   &   94 &  & 27   &   0.04067  &  209   & 98   \\
PG & 2  & 18  &  0.04284   &  179   &   88 &  &  26   &  0.04220   &   183  & 90   \\
CG & 2  &  25 &   0.03017 &   216  &   99 &  &  22   &  0.04173   &  214   &  103  \\ 
& &   &     &     &    &  &     &     &     &    \\ 
A-GCG & 3 & 44 & 0.04601       &   343  &  159  &  &  51  &   0.05319&  387   &  186  \\
GCG & 3  &  57 &   0.04016 &  397   &  184  &  & 72 &  0.04800     &  414   &  198  \\
PG & 3  & 44  &   0.06446  &  417   &  203  &  & 81    &  0.05264   & 484    & 226   \\
CG & 3  &  48 &   0.04083 &  392   &  173  &  & 74  & 0.03348       &  393   &  182  \\ \hline
\end{tabular}
\end{table}

\begin{example} \label{exam2} 
	Consider the multiobjective optimization  given by
		\begin{eqnarray}\label{ex2}
		\nonumber&&\min \: \mathcal{V}(x):=\big(\mathcal{V}_{1}(x),\mathcal{V}_{2}(x),\mathcal{V}_{3}(x)\big)\\&&\hspace{1mm}\mbox{s.t. } \: x\in [-\ell,\ell]^n,
	\end{eqnarray}
 where 
 $$\mathcal{V}_{1}(x)=\frac{1}{n^2}\sum_{j=1}^{n}i(x_i-4)^4,\hspace{5mm} \mathcal{V}_{2}(x)=\exp{\Big(\sum_{j=1}^{n}\frac{x_i}{n}\Big)}+\|x\|^2_2, $$
$$\mathcal{V}_{3}(x)=\frac{1}{n(n+1)}\sum_{j=1}^{n}i(n-i+1)\exp{(-x_i)},$$
and $\ell>0$.
\end{example}
The problem (\ref{ex2}) can be reformulated as (\ref{prob}) where $\mathcal{V}_{i}=f_i+g_i$ for $f_1(x)=\mathcal{V}_{1}(x)$, $g_1(x)=0$, $f_2(x)=\exp{\big(\sum_{j=1}^{n}\frac{x_i}{n}\big)}$, $g_2(x)=\|x\|^2_2$, $f_3(x)=\mathcal{V}_{3}(x)$, $g_3(x)=0$. The comparative performance of {\bf A-GCG}, {\bf GCG}, {\bf PG}, and {\bf CG} in the implementation for (\ref{ex3333}) for different $\ell$ and $n$ value is shown in Figures \ref{Figure3} and Table \ref{Table2}.

\begin{figure}[htb]
\centering
  \begin{tabular}{@{}cccc@{}}
\begin{tabular}{@{}cccc@{}}
\hspace{-2mm}\includegraphics[width=.55\textwidth]{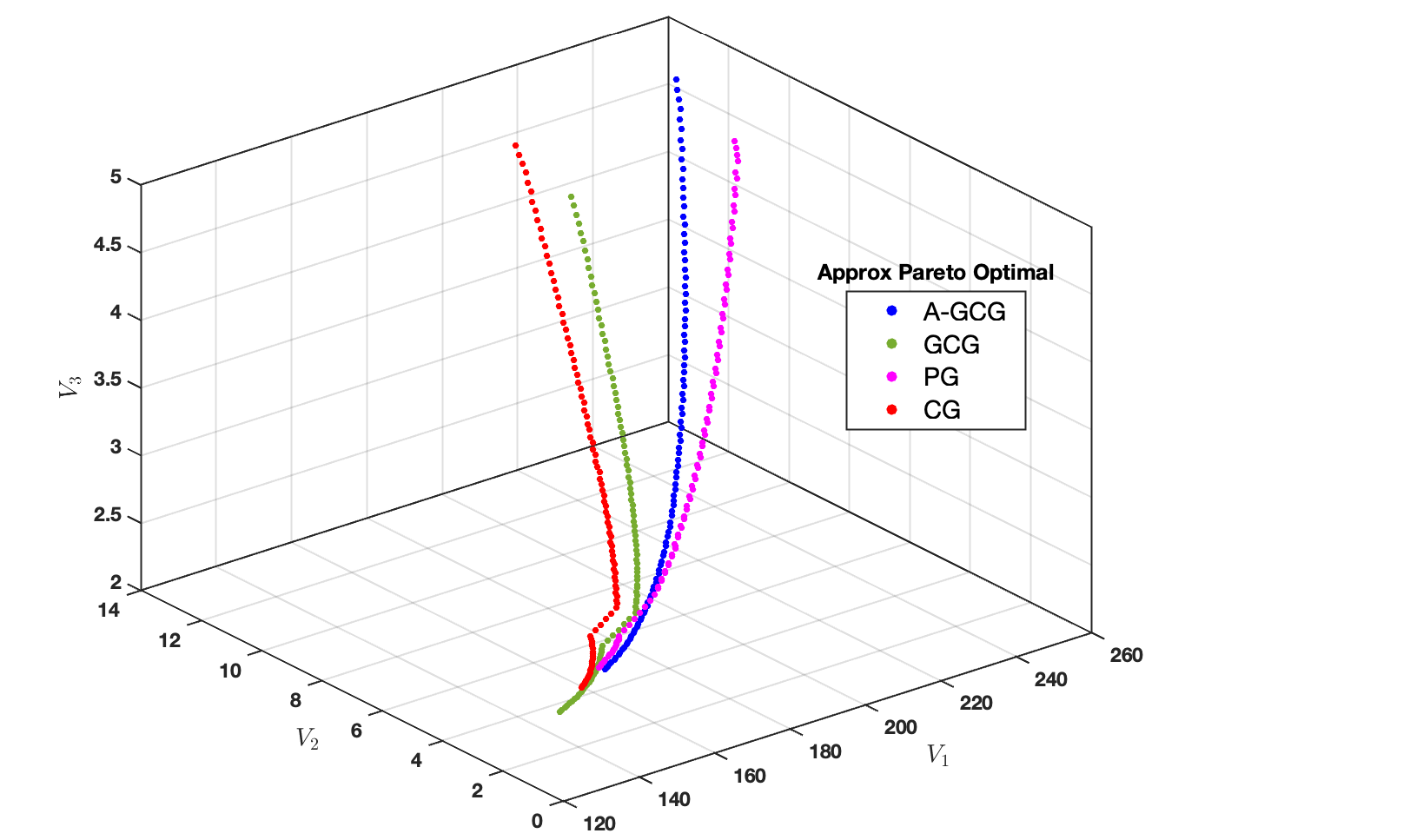} 
\end{tabular}
&\hspace{-16mm}
\begin{tabular}{@{}cccc@{}}
\includegraphics[width=.55\textwidth]{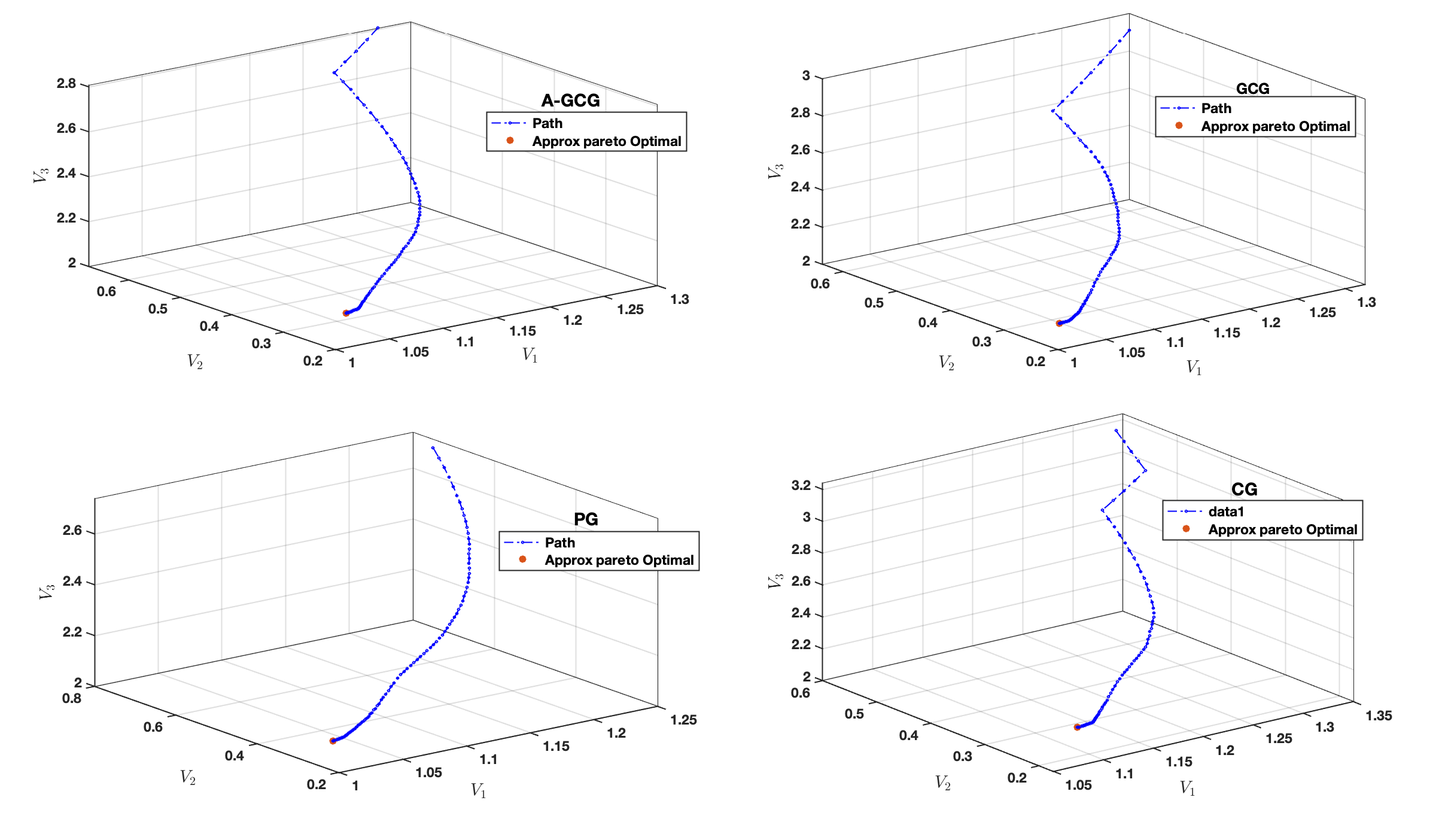} 
\end{tabular}
  \end{tabular}
  \caption{Approximate Pareto optimal and iterative path for Example \ref{exam2} for $n=50$, $\ell=10$, and for the same starting point for each method, and $\mu=10^{-2}$.}\label{Figure3}
\end{figure}

\begin{table}[]
\caption{Example \ref{exam2} for $x_0=2(1,\dots,1)$, $n=50$, and $\mu=10^{-3}$.}\label{Table2}
\begin{tabular}{llllllllll}
\hline
 & \multicolumn{4}{l}{$\ell=10$} &  & \multicolumn{4}{l}{$\ell=100$} \\ \cline{2-5} \cline{7-10} 
Solver & $\#$Iter    &  CPU   &  $\#$Fc   &  $\#$$\nabla$  &  &    $\#$Iter    &  CPU   &  $\#$Fc   &  $\#$$\nabla$    \\ \hline
A-GCG & 41    &   0.08124 & 311 & 159  &  &   27  &   0.12255  &   345  &  175  \\
GCG &   49  &  0.10591   &  375   &  154  &  &   27  &   0.08768  &   489  &  222  \\
PG &  64   &    0.07443 & 408    &  193  &  &    71 &    0.13059 &   576  &  234  \\
CG &  42   &  0.06412 &   360  &  161  &  &   29  &   0.09036  &   411  &  189  \\ \hline
\end{tabular}
\end{table}
\begin{example}\label{exam3}  Consider the
constrained multiobjective optimization problem that appears in portfolio optimization \cite{19qq}:
	\begin{eqnarray}\label{ex3}
	\nonumber&&\min \mathcal{V}(x)=\Big(x^{\top}
	b,x^{\top}Ax\Big)\\&&\hspace{1mm}\mbox{s.t. }  x=(x_{1},\ldots,x_{n})\in\mathbb{R}_{+}^{n}\\&&\hspace{6mm}\sum_{r=1}^{n}x_{r}=1,\nonumber
	\end{eqnarray}
	where $A$ is symmetric positive semidefinite $n\times n$ matrix, and $b\in\mathbb{R}^{n}$. \end{example}   
The problem (\ref{ex3}) can be reformulated as (\ref{prob}) where  $\mathcal{V}_{i}=f_{i}+g_{i}$ for 
$f_{1}(x)=x^{\top}b$,  $f_{2}(x)=x^{\top}Ax$, and 
$g_{i}(x)=\delta_{C\cap\mathbb{R}_{+}^{n}}(x)$ for $i=1,2$, with $\delta_{C\cap\mathbb{R}_{+}^{n}}$ being the indicator function of $C\cap\mathbb{R}_{+}^{n}$ and $C=\{(x_{1},\ldots,x_{n})\in\mathbb{R}_{+}^{n}:\sum_{r=1}^{n}x_{r}=1\}$. 
In our experiments, we use a real data type that occurs in portfolio optimization studies, as seen in \cite{19qq}, and hence we take a symmetric positive semidefinite $n\times n$ matrix $A=[a_{lj}]_{n\times n}$ and a vector $b$ in $\mathbb{R}^{n}$ where $-1\leq a_{lj}\leq 1$ for all $l,j=1,\ldots,n$ and $-2e_{n}\preceq b\preceq 2e_{n}$. Figures~\ref{Figure4} and~\ref{Figure5}, and Table~\ref{Table3} show the achieved computational results of the three algorithms {\bf A-GCG}, {\bf GCG}, and {\bf PG} applied on (\ref{ex3}). 
 \begin{figure}[h] 
	\centering
	\hspace{-1.5mm}\includegraphics[scale=0.30]{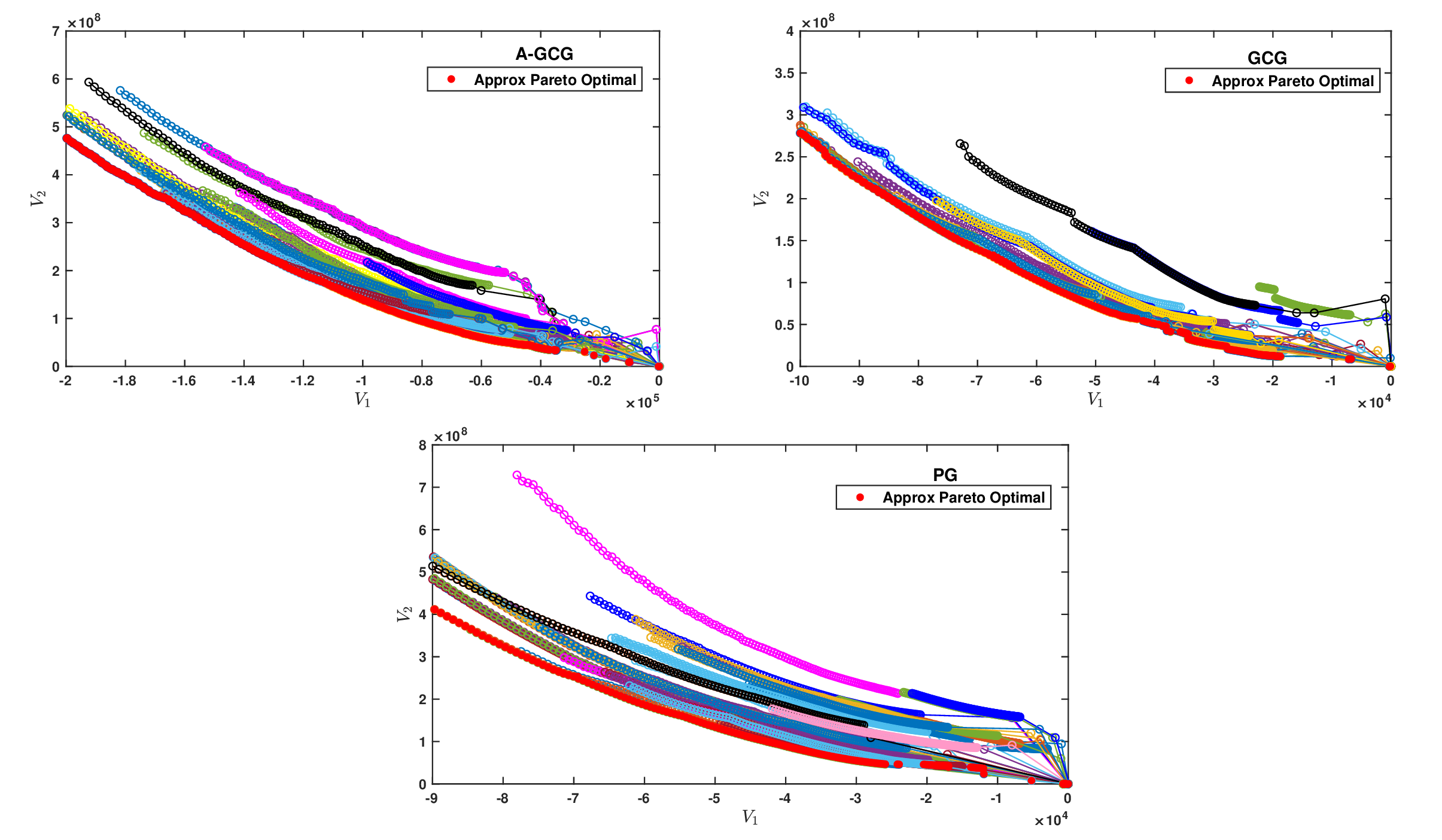} 	
	
	\caption{Iterative path and approximate Pareto optimal obtained for Example \ref{exam3} (for $n=50$) using  200 different starting points in $[-10,10]^n$ and $\mu=10^{-3}$.} \label{Figure4}
\end{figure}
\begin{figure}[htb]
\centering
  \begin{tabular}{@{}cccc@{}}
\hspace{-7mm}\includegraphics[width=.88\textwidth]{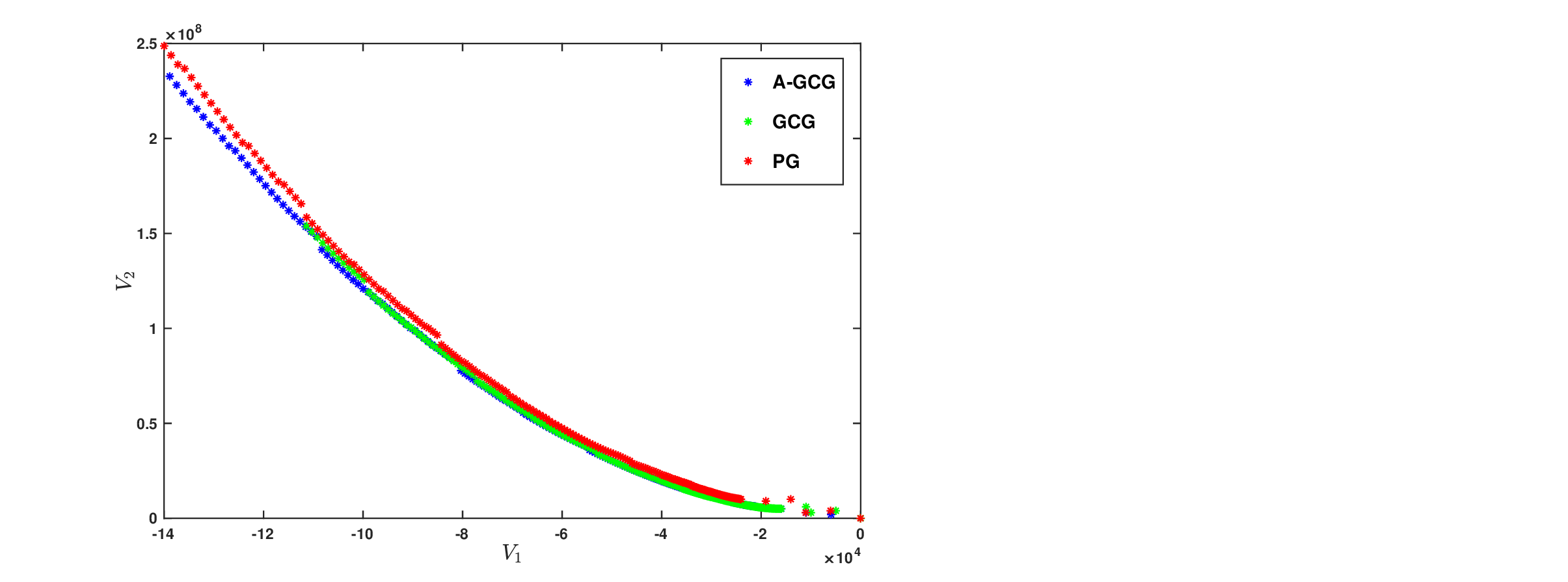} &\hspace{-45mm}
\includegraphics[width=.88\textwidth]{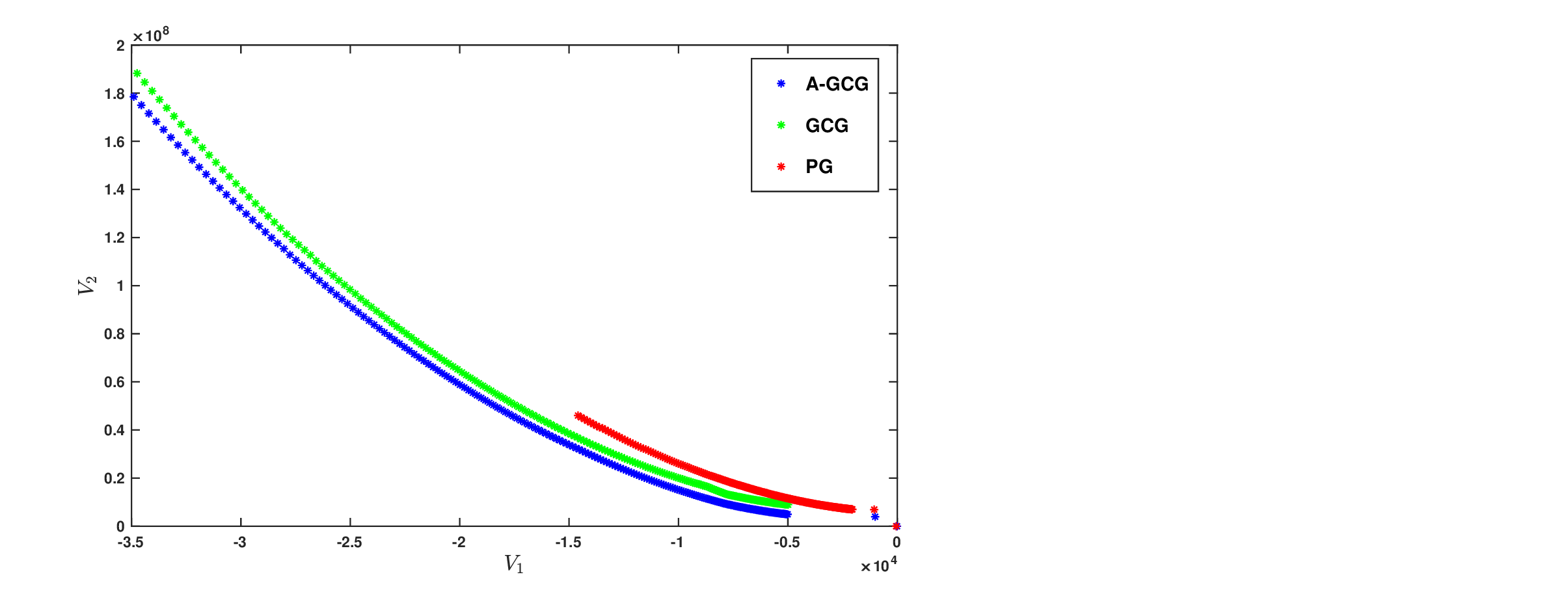} 
    \end{tabular}
  \caption{Iterative path and approximate Pareto optimal obtained for Example \ref{exam3} (for $n=50$) using  200 different starting points in $[-100,100]^n$ and $\mu=10^{-4}$.}\label{Figure5}
\end{figure}

\begin{table}
\caption{Example \ref{exam3} for $x_0=2(1,\dots,1)$, $n=20$, and $\mu=10^{-3}$.}\label{Table3}
\begin{tabular}{llllllllll}
\hline
 & \multicolumn{4}{l}{$n=2$} &  & \multicolumn{4}{l}{$n=10$} \\ \cline{2-5} \cline{7-10} 
 Solver& $\#$Iter    &  Cput   &  $\#$Fc   &  $\#$$\nabla$  &  &    $\#$Iter    &  Cput   &  $\#$Fc   &  $\#$$\nabla$    \\ \hline
A-GCG & 9    &  0.05782&   193  &  103  &  & 16   &  0.07699   &   278  &   128 \\
GCG &  21   &  0.03280   &  321   &  146  &  &    18 &   0.05407 &   316  &  159  \\
PG & 17   &  0.05719   &  379   &  163  &  & 25    &   0.06593 &     392 & 172 \\
CG &  23  &  0.05182   & 309    &  135  &  &    18 &  0.05693   & 372    &  162  \\ \hline
\end{tabular}
\end{table}
The numerical results in the three examples confirm that our algorithm ({\bf A-GCG}) outperforms {\bf GCG}, {\bf PG} and the {\bf CG} in terms of taking less average number of iterations, less average functional value computations, gradient computations, and in finding the $\mu$-approximate Pareto critical point with small number of iterations. However, here our method requires slightly more CPU time than the four algorithms to attain the required stopping criteria. The main reason could be our method included more steps and computations involving both the adaptive and line search procedures, as well as parameters, which allowed us to achieve a better result under generalized assumptions.  We also can see that our algorithm extracted a competently close approximated Pareto optimal points. Computational experiments in Examples~\ref{exam1}, \ref{exam2}, and~\ref{exam3} demonstrated the efficiency, ease of implementation, and suitability of our proposed method for handling constrained multiobjective optimizations.

\section*{Conclusion} 
We suggested a new generalized conditional gradient method that treats the special class of constrained multiobjective optimization problems. The method allows the use of adaptive and line search procedures to produce parameters, which improves the implementation of the method. 
 Detailed computational, iteration complexity, and convergence
  analysis of the proposed method is given under reasonable assumptions, showing that every limit point of the sequence  generated by the algorithm is a Pareto critical
  point, the algorithm generates the first $\mu$-approximate Pareto critical point within
   ${\bf\it {O}}(1/\mu^{2})$ iterations (for $\mu>0$), 
  and the sublinear convergence maintained the best achievable rate ${\bf\it {O}}(1/k)$ ($k$ is number of iterations). Simple numerical experimental results were done to verify the validity of the method.

The ideas of the paper can bring aspects that need to be investigated further about 
multiobjective optimization methods. For example, one can extend it to the bilevel multiobjective optimization case, or consider a method that incorporates both adaptive and line search procedures in Riemannian multiobjective optimization.

\subsection*{Conflict of interest}
Authors have no conflict of interest to declare.




\end{document}